\documentclass[preprint]{elsarticle}
\usepackage{amsfonts,amssymb,,mathrsfs,textcomp,amscd,amsbsy,amsfonts,theorem, nccmath}
\usepackage{verbatim}
\usepackage{enumerate}
\usepackage{amsmath}
\usepackage{geometry}
\usepackage[latin5]{inputenc}
\usepackage{yfonts}

\usepackage{color}
\newtheorem{theorem*}{Theorem}
\theoremstyle{change}
\newtheorem{theorem}{Theorem}[section]
\newtheorem{lemma}[theorem]{Lemma}

\newtheorem{prop}[theorem]{Proposition}
\newtheorem{coro}[theorem]{Corollary}

\theorembodyfont{\rmfamily}  
\theoremstyle{remark}
\newtheorem{nothing}[theorem]{}

\newenvironment{proof}{\noindent{\bf Proof}\ }{\hfill$\square$}

\newcommand{\Aut}{\mathrm{Aut}}

\newcommand{\CC}{\mathbb{C}}

\newcommand{\Def}{\mathrm{Def}}

\newcommand{\End}{\mathrm{End}}

\newcommand{\Hom}{\mathrm{Hom}}

\newcommand{\id}{\mathrm{id}}

\newcommand{\ind}{\mathrm{Ind}}
\newcommand{\Ind}{\mathrm{Ind}}
\newcommand{\infl}{\mathrm{Inf}}
\newcommand{\Inf}{\mathrm{Inf}}

\newcommand{\Irr}{\mathrm{Irr}}
\newcommand{\Isom}{\mathrm{Iso}}

\newcommand{\lexp}[2]{\setbox0=\hbox{$#2$} \setbox1=\vbox to
                 \ht0{}\,\box1^{#1}\!#2}

\newcommand{\lin}{\mathrm{Lin}}

\newcommand{\res}{\mathrm{Res}}
\newcommand{\Res}{\mathrm{Res}}
\newcommand{\Out}{\mathrm{Out}}
\newcommand{\out}{\mathrm{Out}}

\newcommand{\tw}{\mathrm{Tw}}

\newcommand{\ZZ}{\mathbb{Z}}

\newcommand{\Ctimes}{\mathbb{C}^{\times}}
\newcommand{\C}{\mathbb{C}}

\newcommand{\Tw}{\mathrm{Tw}}

\newcommand{\crc}{\mathbb{C}R_{\mathbb{C}}}

\newcommand{\zm}{\mathbb{Z}/m \mathbb{Z}}
\newcommand{\zmm}{\mathbb{Z}_m}
\newcommand{\znn}{\mathbb{Z}_n}
\newcommand{\zn}{\mathbb{Z}/n \mathbb{Z}}

\newcommand\remO[1]{\textcolor{blue}{\textbf{#1}}}

\makeindex

\begin{document}
\begin{frontmatter}

\title{The functor of complex characters of finite groups}
\author{Mehmet Arslan}
\ead{mehmet$\_$arsln@yahoo.com}
\address{MEF University, Faculty of Economics, Maslak, Istanbul, Turkey}

\author{Olcay Co\c{s}kun\corref{cor1}}
\ead{olcay.coskun@boun.edu.tr}
\address{Bogazici University, Department of Mathematics, Bebek, Istanbul, Turkey\\ \& \\
Bogazici University Feza Gursey Center for Physics and Mathematics, Kandilli, Istanbul, Turkey}

\cortext[cor1]{Corresponding author}
\fntext[fn1]{Both authors are supported by Tubitak-1001-113F240.}

\begin{abstract}
We determine the structure of the fibered biset functor sending a finite group $G$ to the complex vector space of 
complex valued class functions of $G$. Previously, it is studied as a biset functor by Bouc and as a 
$\mathbb C^\times$-fibered biset functor by Boltje and the second author. In this paper, we complete the study by  
considering the other choices of the fiber group. As a corollary, we obtain a (previously known) set of composition
factors of the biset functor of $p$-permutation modules. 
\end{abstract}



\end{frontmatter}

\section{Introduction}\label{sec:intro}

One of the fundamental constructions in representation theory of finite groups is the ring of characters of a group.
 Over a 
field of characteristic zero, it determines all representations. The functorial structure of 
character rings is studied in several ways, including Thvenaz and Webb's  description \cite{TW} of its Mackey 
functor structure, and Bouc's description \cite{bouc} of its biset functor structure. Recently, it is 
investigated by Romero \cite{romero} as a Green biset functor and by Boltje and the second author \cite{bolcay} as a  
$\mathbb{C}^\times$-fibered biset functor. In all these cases, it turns out to be a semi-simple functor. 
In the present paper, we are aiming to determine its structure as a fibered biset functor for some other choices of 
fiber groups $A < \mathbb{C}^\times$. 

Fibered biset functors are introduced by Boltje and the second author in \cite{bolcay} as a general framework
for structures having actions of monomial bimodules. It originates from Bouc's theory of biset functors 
\cite{bouc} where the basic tools are bisets and permutation bimodules. Basic theory of fibered bisets and
fibered biset functors together with a parametrization of simple functors is described in \cite{bolcay}. One of 
the most natural examples of fibered biset functors is the functor $R_\mathbb C$ of complex characters 
under tensor product by monomial bimodules. It is shown in \cite{bolcay} that $\CC R_\CC$ is a $\CC^\times$-fibered biset functor, and hence by restriction, it is an $A$-fibered biset functor for any 
subgroup $A\le \CC^\times$.

In order to determine the $A$-fibered structure of $\CC R_C$, we first analyze the structure of $A$-fibered 
bisets in some special cases. In \cite{bouc}, Bouc proves that any transitive biset is a product of five basic bisets, one corresponding to each
one of common operations from module theory, namely, induction, inflation, transport of structure, deflation 
and restriction. In \cite{bolcay}, Boltje and the second author show that for a large fiber
group $A$, a similar decomposition, with new basic elements, is possible for fibered bisets.  Later in 
\cite{CY}, the second author and Ylmaz show that for abelian groups, this decomposition is simpler. One
of our main results, Theorem \ref{thm:factorization}, is a new decomposition theorem which works for $A$-fibered 
bisets over cyclic groups when $A$ is a small fiber group. 

The problem of determining the fibered biset functor structure of $\CC R_\CC$ falls into three cases depending on the fiber group. For the first case, we consider a large fiber group for which the decomposition
for fibered bisets given in  \cite{bolcay} holds. More precisely, we fix a set $\pi$ of
primes and let $A$ be the group of all $p^n$-th roots of unity for all $n \in \ZZ^+$ and for all $p\in \pi$. In this case,
the functor $\CC R_\CC$ turns out to be semisimple. The simple summands are given by the following theorem. 
\begin{theorem*}\label{thm:intromain1}
Let $A$ be as above. There is an isomorphism of $A$-fibered biset functors
\[
\CC R_\CC \cong  \bigoplus_{\substack{(m,\zeta)\in [\Upsilon_\pi}]} S^A_{\mathbb Z/m\mathbb Z, 1, 1, \C_\zeta}.
\]
where $[\Upsilon_\pi]$ denote the set of all pairs $(m, \zeta)$ such that $m\in\mathbb N$ is a $\pi'$-number
and $\zeta$ is a primitive character of $\Aut(\zm)$. Also given $(m, \zeta)\in[\Upsilon_\pi]$, we denote by $S^A_{\mathbb Z/m\mathbb Z, 1, 1, \C_\zeta}$ the simple $A$-fibered biset functor with minimal group $\zm$ and evaluation $\C_\zeta$ at $\zm$. 
\end{theorem*}
As an immediate corollary of this theorem, we recover a set of composition factors for the biset functor of 
$p$-permutation modules. See Corollary \ref{cor:p-perm}. A more general set is already determined by Baumann in \cite{Bau}. 

We also consider the case where $A$ is a finite cyclic $p$-group. In this case, the functor $\CC R_\CC$ is
still semisimple. However 
determination of simple summands is more involved since now there
are new basic $A$-fibered bisets which do not appear in the previous case. We defer the detailed notation to Section 
\ref{sec:mainR}
and state the following theorem.
\begin{theorem*}\label{thm:intromain2}
Let $A$ be a finite cyclic $p$-group. There is an isomorphism of $A$-fibered biset functors
\[
\CC R_\CC \cong \bigoplus_{(m,\zeta)\in[\Upsilon_p]} S_{{\zm, 1, 1, \mathbb C_\zeta}}^A \oplus \bigoplus_{(m,\zeta)\in[\Upsilon_{>|A|}]} S_{\zm, A, \alpha, \mathbb C_\zeta^\alpha}^A.
\]
where $[\Upsilon_p]$ is defined as above, and $\Upsilon_{>|A|}$ denotes the set of all pairs $(m, \zeta)$ as 
above with $m_p > |A|$ and $[\Upsilon_{>|A|}]$ is a full set of representatives of $A$-equivalence classes described in Section \ref{sec:mainR}. 
\end{theorem*}
The third and more general case where the fiber group is finite cyclic group of composite order or a product of a finite 
group and an infinite cyclic group as in the first case can be treated similarly.

The paper is organized as follows. In Section \ref{sec:prelim}, we include basic notations and results on fibered biset functors. Section \ref{sec:decomposition} contains our main results on decomposition of fibered bisets over cyclic groups when the fiber group is small. Using these results, we interpret the parametrization of simple fibered biset functors with cyclic minimal group in Section \ref{sec:paramsimple}. We study actions of fibered bisets on characters in Section \ref{sec:main}. The proofs of our main theorems are contained in the final section, Section \ref{sec:mainR}.

We now give some notations that are valid throughout the paper. We let $G$ and $H$ be finite groups and $A$ be a multiplicatively written (not necessarily finite) abelian group. We
also set
\begin{displaymath}
G^{A} := \textrm{Hom}(G,A) \quad \mathrm{and} \quad G^{*} := \textrm{Hom}(G,\mathbb{C})
\end{displaymath}
and view them as abelian groups with point-wise multiplication.  

\section{Preliminaries}\label{sec:prelim}
In this section we collect  necessary definitions and results from \cite{bolcay}.  A set $X$ is called an \textit{$A$-fibered 
$G$-set} if $X$ 
is an $A\times G$-set such that the action of $A$ is free with finitely many orbits. The category of $A$-fibered 
$G$-sets together with $(A,G)$-equivariant functions is denoted by ${}_G\mathrm{set}^A$. 
Disjoint union of sets induces a coproduct on the category ${}_G\mathrm{set}^A$.  The 
Grothendieck group $B^A(G)$ of this category is called the \textit{$A$-fibered Burnside group} of $G$.

An  $A$-fibered $G$-set is called \textit{transitive} if the $G$-action on the set of $A$-orbits is 
transitive. It is easy to show that there is a bijective correspondence
between isomorphism classes of transitive $A$-fibered $G$-sets  $X$ and $G$-conjugacy classes of 
pairs $(U,\phi)$ where $U$ is a subgroup of $G$ and $\phi: U\rightarrow A$ is a group homomorphism. The 
bijection is given by associating $X$ to $(U, \phi)$ if $U$ is the stabilizer of some $A$-orbit in $X$ and $U$ acts on 
this $A$-free orbit via $\phi$. We call the pair $(U,\phi)$ corresponding to $X$ the \textit{stabilizing pair} of $X$. We denote by $\mathcal M_G(A)$ the set of all such pairs $(U, \phi)$, and write $[U,\phi]_G$ for the isomorphism class of the $A$-fibered $G$-set with the stabilizing pair $(U,\phi)$. The group $G$ acts on $\mathcal M_G(A)$
via conjugation. We write $\mathcal M_G^G(A)$ or $\mathcal M_G^G$ (whenever $A$ is clear from the content) for the set of $G$-fixed points in $\mathcal M_G(A)$.

The set 
$\mathcal M_G(A)$ has a poset structure given by the ordering $(K, \kappa) \le (L, \lambda)$ if $K \le L$ and $\kappa=\lambda|_{K}$. Together with the above $G$-action, it becomes a $G$-poset.

We further write ${}_G\mathrm{set}_H^A$ for the category of $A$-fibered 
$G\times H$-sets. By the usual convention, we regard any object in this category as an $A$-fibered $(G,H)$-biset.  We write $\big[\frac{G\times H}{U,\phi}\big]$ instead of 
$[U,\phi]_{G\times H}$ wherever we regard it as a fibered biset. With this notation, any ordinary biset 
$\big[\frac{G\times H}{U} \big]$ is regarded as a fibered biset as  $\big[\frac{G\times H}{U,1} \big]$. Here $1$ 
denotes the trivial homomorphism from $U$ to $A$.

Let $K$ be another finite group, $X$ be an $A$-fibered $(G,H)$-biset and $Y$ be an $A$-fibered $(H,K)$-biset.
We write $AH = A\times H$. The usual amalgamated product of $X$ and $Y$ over $AH$ is the set 
$X\times_{AH} Y$ of $AH$-orbits in $X\times Y$. Here $AH$ acts on $X\times Y$ via
\[
(a,h)\cdot (x,y) = (x\cdot (a^{-1},h^{-1}),(a,h)\cdot h)
\]
for any $(a,h)\in A\times H$ and $(x,y)\in X\times Y$. Given $(x,y)\in X\times Y$, we write $(x,_{AH} y)$ for its 
image in $X\times_{AH} Y$. Given $(a, g, k)\in A\times G\times K$, we define
\[
(a,g)\cdot (x,_{AH} y)\cdot k = (agx,_{AH} yk)
\]
and with this action $X\times_{AH} Y$ becomes an $(AG, K)$-biset. In general, this is not an $A$-free set. We
define the tensor product $X\otimes_{AH} Y$ of $X$ and $Y$ as the subset of $X\times_{AH} Y$ consisting 
of $A$-free orbits which is an $A$-fibered $(G,K)$-biset. We denote a free orbit
in $X\times_{AH} Y$ by $x\otimes y$. Note that the tensor product of $A$-fibered bisets is linear in both coordinates
and induces a bilinear map
\[
B^A(G,H)\times B^A(H,K)\to B^A(G,K).
\]

Next we recall the product formula of two transitive $A$-fibered bisets.
By Goursat's Theorem, there is a bijective correspondence between subgroups $U$ of $G\times H$ and 
quintuples $(P, K, \eta, L, Q)$ given by  $P = p_1(U)$ and $Q= p_2(U)$, the first and the 
second projections of $U$. Also $K = k_1(U) = p_1(U\cap (G\times 1))$ and $L = k_2(U) = p_2(U\cap (1\times H))$. 
We clearly have that $K\unlhd P$ and $L\unlhd Q$ and the subgroup $U\le G\times H$ induces an isomorphism
$\eta: Q/L\to P/K$ given by $\eta(hL) = gK$ if $(g,h)\in U$. We call this the \textit{Goursat correspondence} and $U$ 
the Goursat correspondent of  $(p_1(U), k_1(U), \eta, k_2(U), p_2(U))$ and vice versa.

Given a pair $(U, \phi)\in \mathcal M_{G\times H}(A)$, we also write $\phi |_{K\times L} = \phi_1\times 
\phi_2^{-1}$. We call the triple $l(U,\phi) = (P, K, \phi_1)$ (resp. $r(U,\phi) = (Q, L, \phi_2)$) the left invariants
(resp. right invariants) of $(U, \phi)$. We sometimes shorten the invariants and write $l_0(U, \phi) = (K, \phi_1)$
and $r_0(U, \phi) = (L, \phi_2)$.

With this notation, given transitive elements $\big[\frac{G\times H}{U,\phi}\big]$ and $\big[\frac{H\times K}{V,\psi}\big]$, by \cite[Corollary 2.5]{bolcay} we have
\[
\Big[\frac{G\times H}{U,\phi}\Big]\otimes_{AH} \Big[\frac{H\times K}{V,\psi}\Big] = \sum_{\substack{x\in p_2(U)
\backslash H/p_1(V)\\ \phi_2\vert_{H_x} = \lexp{x}\psi_{1}\vert_{H_x}}} \Big[\frac{G\times K}{U\ast \lexp{(x,1)}V,\phi\ast\lexp{(x,1)}\psi}\Big]
\]
where $H_x = k_2(U)\cap\lexp{x}k_1(V)$, and the subgroup $U\ast V$ is the composition 
\[
U\ast V = \{ (g,k)\in G\times K\, \vert\, (g,h)\in U, (h,k)\in V \, \mbox{\rm for some}\, h\in H \}
\]
and the homomorphism $\phi\ast \psi: U\ast V\to A$ is defined by 
\[
(\phi\ast\psi)(g,k) = \phi(g,h)\cdot\psi(h,k)
\]
for some choice of $h\in H$ such that $(g,h)\in U$ and $(h,k)\in V$. Note that the homomorphism $\phi\ast\psi$ is independent of 
the choice of $h\in H$. This product is called the \textit{Mackey product} of fibered bisets.

Now we define fibered biset functors. Let $A$ be an abelian group and $R$ be a commutative ring with unity. Let $\mathcal C:= \mathcal C_R^A$ be the category where
\begin{enumerate}
\item[(i)] the objects of $\mathcal C$ are all finite groups.
\item[(ii)] Given two finite groups $G$ and $H$, $$\Hom_{\mathcal C}(G,H) := R\otimes B^A(H, G) = RB^A(H, G).
$$
\item[(iii)] The composition is the $R$-linear extension of the tensor product of $A$-fibered bisets introduced above.
\end{enumerate}

An \emph{$A$-fibered biset functor over $R$} is an $R$-linear functor $\mathcal C\to {}_R$mod. The class
of all $A$-fibered biset functors together with natural transformations between them forms a category, denoted 
by $\mathcal F:=\mathcal F_R^A$. Since ${}_R$mod is an abelian category, the category $\mathcal F$ is also 
abelian. Classification of simple fibered biset functors is done in \cite[Section 9]{bolcay}. We review 
the parametrization in a special case in Section \ref{sec:paramsimple}.

\section{Fibered bisets for cyclic groups}\label{sec:decomposition}
By \cite[Corollary 7.3.5]{bouc}, one parameter for the classification of simple summands of the biset functor of character rings is finite cyclic groups. It turns out that, also in the case of fibered biset functors, one of the parameters runs over a set of finite cyclic groups. In this section, we specialize some results from \cite{bolcay} to the case of cyclic groups.
 
First we introduce notation for basic fibered bisets which is used throughout the paper. Let $H$ be a subgroup of $G$ 
and $N$ be a normal subgroup of $G$. Also let $G^\prime$ be a finite group and $\lambda: G' \to G$ be 
an isomorphism. 
  
We define \emph{induction} from $H$ to $G$ and \emph{restriction} from $G$ to $H$ as the transitive bisets given,
respectively, by
\[
\ind_H^G := {}_GG_H\quad \mathrm{and}\quad\res^G_H := {}_HG_G
\]
where we regard the set $G$ as a $(G,H)$-biset (resp. as an $(H,G)$-biset) in the usual way, via left and right 
multiplication by the corresponding group. We also define \emph{deflation} from $G$ to $G/N$ and \emph{inflation} from $G/N$ 
to $G$ as the transitive bisets given, respectively, by
\[
\Def_{G/N}^G := {}_{G/N}(G/N)_G\quad \mathrm{and}\quad\infl_{G/N}^G:= {}_G(G/N)_{G/N}.
\]
As above, we regard the set $G/N$ as a $(G/N,G)$-biset (and as a $(G,G/N)$-biset) in the usual way. 
Finally, we 
define the \emph{transport of structure} from $G'$ to $G$ through $\lambda$ as the transitive biset given by
\[
\mbox{\rm c}_{G,G'}^\lambda=\Isom^\lambda_{G,G'}:= {}_GG_{G'}
\]
where the $G$ action is the left multiplication and the $G^\prime$-action is multiplication through $\lambda$.
In all these cases, the $A$-action is trivial.

There are three other basic elements that we need in this paper. The first is twist biset defined as follows. Let 
$\phi\in G^A$ be a homomorphism from $G$ to $A$. Then the \textit{twist} by $\phi$ at $G$ is the $A$-fibered 
$(G,G)$-biset  
$$\tw_G^\phi=\Big( \frac{G\times G}{\Delta(G), \Delta_\phi}\Big)$$
where $\Delta(G)$ is the diagonal subgroup of $G\times G$ and $\Delta_\phi(g,g) = \phi(g)$ for any $g\in G$.

The other two basic elements are both idempotents in $B^A(G,G)$. Let $(K,\kappa)\in\mathcal M_G^G(A)$. We write 
$\Delta_K(G) = \{(g,h)\in G\times G : gK = hK\}$ and define $\phi_\kappa: \Delta_K(G)\rightarrow A$
by $\phi_\kappa(g,h) = \kappa(gh^{-1}) = \kappa(h^{-1}g)$. With this notation, following \cite{bolcay}, we define
\[
E_{(K,\kappa)}^G = \Big(\frac{G\times G}{\Delta_K(G), \phi_\kappa} \Big).
\]
For simplicity, we write $e_{(K,\kappa)}^G$ for the isomorphism class of $E_{(K,\kappa)}^G$. Whenever there is
no risk of confusion we drop the letter $G$ and write $e_{(K, \kappa)}$. As shown in \cite[Section 4.3]{bolcay}, for all $(K,\kappa), (L, \lambda)\in\mathcal M_G^G(A)$, we have
\[
e_{(K,\kappa)} \cdot e_{(L,\lambda)} = \begin{cases} 
      e_{(KL, \kappa\lambda)} & \mathrm{if} \,\, \kappa|_{K\cap L} = \lambda|_{K\cap L} \\
      0 &  \mathrm{otherwise.}
   \end{cases}
\]
The following construction introduces orthogonal idempotents from $e_{(K, \kappa)}$. Let $\mu^{\triangleleft}_{K,L}$
denote the Mobius coefficient with respect to the poset of normal subgroups of $G$. Given $(K, \kappa)\in 
\mathcal M_G^G$, we define
\[
f_{(K,\kappa)} = \sum_{(K,\kappa)\le (L, \lambda)\in M_G^G}\mu^\triangleleft_{K, L} e_{(L, \lambda)}. 
\]
By \cite[Proposition 4.4]{bolcay}, the idempotents $f_{(K,\kappa)}$ form a set of mutually orthogonal idempotents in $B^A(G,G)$ summing up to the identity element $e_{(1,1)}$. For further properties of these elements, we refer to
Section 4 of \cite{bolcay}.

In \cite{bouc}, Bouc showed that any transitive $(G, H)$-biset $(G\times H)/U$ is equal to a Mackey product of the 
above five basic bisets. More precisely, we have  
\begin{align*}
 \left(\frac{G\times H}{U}\right)=\Ind_{P}^{G}\, \Inf_{P/K}^{P}{\Isom}_{P/K,Q/L}^{\eta}\Def_{Q/L}^{Q}\Res_{Q}^{H}
\end{align*}
where $(P, K, \eta, L, Q)$ is the Goursat correspondent of $U$. A similar, but partial, decomposition for a transitive 
$A$-fibered $(G,H)$-biset is given in \cite[Proposition 2.8]{bolcay}. Particularly if $\phi: U\to A$, then we have
\begin{equation}\label{eqn:partial}
 \left(\frac{G\times H}{U,\phi}\right)=\Ind_{P}^{G} \, \Inf_{P/\hat{K}}^{P}\otimes_{AP/\hat{K}}Y\otimes_{AQ/\hat{L}}
\Def_{Q/\hat{L}}^{Q}\Res_{Q}^{H}
\end{equation}
where $\hat{K}$ and $\hat{L}$ are kernels of $\phi_{1}$ and $\phi_{2}$, respectively. Here $Y$ is a transitive
$A$-fibered $(P/\hat{K},Q/\hat{L})$-biset with full left and right projections and faithful left and right restrictions 
of the induced fiber homomorphism $\phi^\prime$. A further decomposition of $Y$ is given in 
\cite[Section 10]{bolcay} when $A$ satisfies a condition. We recall the condition.

\begin{nothing} {\bf Hypothesis.}\label{hypo}
There is a (unique) set $\pi$ of primes such that for every $n\in \mathbb N$, the $n$-torsion part of $A$ is cyclic of order $n_\pi$,
where $n_\pi$ denotes the $\pi$-part of $n$.
\end{nothing}

In \cite{CY}, a simplified decomposition for transitive $A$-fibered $(G,H)$-bisets is given when 
$A$ satisfies the above hypothesis and $G$ and $H$ are abelian groups. We recall this decomposition. Let 
$(U,\phi)\in \mathcal M_{G\times H}(A)$. We write $(P,K,\eta,Q,L)$ for the Goursat correspondent of $U$. We also write 
$\tilde \phi = \tilde\phi_1\times \tilde\phi_2$ for an extension of $\phi$ to $P\times Q$ which exists since the group $P\times Q$ is 
abelian and $A$ is divisible (by the hypothesis). Then by \cite[Theorem 1]{CY}, we have

\begin{equation}\label{thm1:decomposition}
\Big(\frac{G\times H}{U,\phi}\Big) \cong \Ind_{P}^G \tw_{P}^{\tilde \phi_1}\Inf_{P/K}^{P}
\Isom_{P/K,Q/L}^\eta\Def^{Q}_{Q/L}\tw^{\tilde \phi_2}_{Q}\Res^H_{Q}.
\end{equation}
Note that, this decomposition is valid whenever $|A|_p\ge \mathrm{max}(|G|_p, |H|_p)$ for any prime number $p$ since the extension $\tilde\phi$ would still exist under this last condition.
For the rest of this section, we determine the 
decomposition of $A$-fibered $(G,H)$-bisets when $G, H$ and $A$ are all finite cyclic groups.

By the next lemma, we reduce this problem to $p$-groups. This is the fibered 
version of \cite[Proposition 2.5.14]{bouc}.

\begin{lemma}\label{lem:direct-product}
Let $X$ be an $A$-fibered $G$-set and $Y$ be an $A$-fibered $H$-set. Then the Cartesian product $X\times Y$ is an $A$-fibered 
$G\times H$-set via
\[
(g, h, a) \cdot (x,y) = (a\cdot g\cdot x, a\cdot h\cdot y)
\]
for all $(g, h, a)\in G\times H\times A$ and $(x, y)\in X\times Y$. Moreover the correspondence $(X, Y)\mapsto X\times Y$ induces a bilinear
map from $B^A(G)\times B^A(H)$ to $B^A(G\times H)$ and hence a group homomorphism 
\[
B^A(G)\otimes_\mathbb Z B^A(H) \to B^A(G\times H).
\]
This is a unital injective ring homomorphism which becomes an isomorphism when $G$ and $H$ are of coprime orders.
\end{lemma}
Proof of this lemma is almost identical to the proofs of the results mentioned above. The only difference is
that one needs to check that the Cartesian product $X\times Y$ is free as an $A$-set under the diagonal action $a\cdot (x, y) = (a\cdot x, 
a\cdot y)$ which clearly holds. We skip the details of the proof.

Now given an abelian group $G$, we decompose it as
\[
G \cong \prod_{p} G_p
\] 
where the product is over the set of distinct prime divisors of $|G|$ and for a prime divisor $p$, we write $G_p$ for the Sylow $p$-subgroup
of $G$. Therefore, by the above lemma, given any element $X\in B^A(G, G)$, we can decompose it as 
\[
X \cong \prod_{p} X_p
\]
where $X_p\in B^A(G_p, G_p)$ for each prime divisor $p$ of $|G|$. Also it is easy to show that if $H$ is abelian of
 order less than $|G|$ and $X \cong Y\otimes_{AH} Z$, then we also have
\[
X\cong \prod_p Y_p\otimes_{AH_p} Z_p
\]
where $H_p$ is the Sylow $p$-subgroup of $H$ and $Y_p$ and $Z_p$ are chosen as above. Therefore the decomposition can be done by considering one prime at a time.
Notice that if $A$ has trivial $p$-torsion then 
$B^A(G_p, G_p) \cong B(G_p, G_p)$ and hence in this case, any $A$-fibered $(G_p, G_p)$-biset may be regarded as an ordinary $(G_p, G_p)$-biset and hence can be decomposed using Bouc's result. If $|A|_p\ge |G|_p$, then we can use Equation \ref{thm1:decomposition} to decompose $X_p$. Hence, for abelian groups $G$ and $H$, to determine a factorization of an $A$-fibered $(G,H)$-biset $X$, we only need to determine the case where the groups $G, H$ and $A$ are $p$-groups for a fixed prime $p$ and $|A| < |G|\ge |H|$. Then it turns 
out that the decomposition depends on the order of $|A|$ relative to $|H|$.

More precisely, we fix $l\in \mathbb N$ and let $A = \langle a\rangle$ be a cyclic group of 
order $p^l$, written multiplicatively. Also let $G$ and $H$ be cyclic groups of orders $p^m$ and $p^n$, 
respectively. We assume that $l < m\ge n$. Note that the case $m\le n$ can be done similarly or by taking opposites.

Let $(U, \phi)\in \mathcal M_{G\times H}(A)$. By \cite[Proposition 2.8]{bolcay}, we may also assume that 
\[
p_1(U) = G,\,\,\, p_2(U) = H,\,\,\, \ker (\phi_1) = 1,\,\,\, \ker (\phi_2) = 1.
\]
Since $A$ is finite, the above assumptions put some conditions on $(U, \phi)$. Indeed, to have a faithful homomorphism 
$\phi_1 : k_1(U) \to A$, we should have $|k_1(U)| \le |A|$. Similarly, for $k_2(U)$. Thus, if $(U, \phi)\in 
\mathcal M_{G\times H}(A)$ satisfies the above conditions, we also have 
\[
|k_1 (U)|\le |A| \ge |k_2(U)|.
\]
Now given a pair $(U, \phi)$ as above, we let $(G, K, \eta, L, H)$ be its Goursat correspondent. We fix a 
generator $h$ of $H$ and let $g$ be a generator of $G$ such that $\eta (hL) = gK$. This is possible since $p_1(U) 
= G$. Note that with this choice, we get $(g,h)\in U$, and hence $\Delta_\eta(G) :=\langle(g, h)\rangle\le U$. 
Notice also that 
$|\Delta_\eta(G) \cap (1\times L)| =1$
and $|\Delta_\eta(G)\cdot(1\times L)| = |G|\cdot |L| = |U|$. Hence we have
\[
U = \Delta_\eta(G) \cdot (1\times L) \quad \mathrm{and}\quad G\times H = \Delta_\eta(G)\cdot(1\times H).
\]
We further write
\[
\phi = \phi_\eta\cdot\phi_L
\]
where $\phi_\eta$ (resp. $\phi_L$) is the restriction of $\phi$ to $\Delta_\eta(G)$ (resp. to $1\times L$). Note that, 
$\phi_L = 1\times \phi_2$ and hence it is faithful by our assumption. 

Next suppose $|A|\ge |H|$, then $\phi_L$ extends to a homomorphism $\phi_H: 1\times H=H \to 
A$. 
Given $(x, y)\in G\times H$, we write $(x, y) = (x, \tilde x)\cdot (1, (\tilde x)^{-1}y)$ as an element of 
$\Delta_\eta(G)\cdot(1\times H)$. Define $\psi : G\times H\to A$ by 
$\psi(x, y) = \phi_\eta(x, \tilde x)\phi_H((\tilde x)^{-1}y)$. This is a homomorphism since both $\phi_\eta$ and
$\phi_H$ are homomorphisms and $\widetilde{xt} = \tilde x\tilde t$ for any $x, t\in G$. Also $\psi$ is
an extension of $\phi$. Indeed for any $x\in G$, we have $(x, \tilde x)\in U$ by definition, and hence if $(x, y)$ is also in 
$U$, then $(1, (\tilde x)^{-1}y)\in U$.
In particular, we can apply Equation \ref{thm1:decomposition} and get
\[
\Big( \frac{G\times H}{U, \phi}\Big) \cong \tw_G^{\psi_G}\infl_{G/K}^G\Isom_{G/K, H/L}^{\eta}\Def_{H/L}^H\tw_H^{\phi_H}.
\]
Here, we write $\psi_G$ for the restriction of $\psi$ to $G = G\times 1$. Note that the term $\tw_G^{\psi_G}$ may be
decomposed further, but for our aims, this version is sufficient.

For the second case where $|A| < |H|$, by \cite[Proposition 4.2.(b)]{bolcay}, there is an isomorphism 
\begin{equation}\label{eqn:finite1st}
\Big( \frac{G\times H}{U, \phi}\Big) = \Big( \frac{G\times H}{\Delta_\eta(G)\cdot(1\times L), \phi_\eta\cdot \phi_L}\Big) \cong
\Big( \frac{G\times H}{\Delta_\eta(G), \phi_\eta}\Big) \otimes_{AH} E_{(L,\phi_L)}^H.
\end{equation}
of $A$-fibered $(G,H)$-bisets. We consider the factorization of the terms in the above isomorphism separately. First consider the pair $(\Delta_\eta(G),
 \phi_\eta)\in \mathcal M_{G\times H}$. We clearly have
\[
p_1(\Delta_\eta(G)) = G,\,\,\, p_2(\Delta_\eta(G)) = H\,\,\, \mbox{\rm and}\,\,\, k_2(\Delta_\eta(G)) = 1.
\]
In general, $k_1(\Delta_\eta(G)) \neq K$, and we let $K_\eta = k_1(\Delta_\eta(G))$, and 
\[
\tilde\eta: H\to G/K_\eta
\]
be the canonical isomorphism determined by $\Delta_\eta(G)$. Also let $\tilde\phi: G\to A$ be the character given by 
$\tilde\phi(g) = \phi_\eta(g, h)$. With this notation, we have 
\[
\Big( \frac{G\times H}{\Delta_\eta(G), \phi_\eta}\Big) = \tw^{\tilde\phi}_G \infl_{G/K_\eta}^G \Isom_{G/K_\eta, H}^{\tilde\eta}.
\]
Indeed, to prove the equality, write $(V, \psi)$ for the stabilizing pair of the right hand side. Then 
\begin{eqnarray*}
V &=& \Delta(G)\ast \{(x, xK_\eta): x\in G\}\ast \{ (\tilde\eta(y), y): y\in H \} \\ 
&=& \{(x, y)\in G\times H: xK_\eta = \tilde\eta(y)\} = \Delta_\eta(G).
\end{eqnarray*}
On the other hand, since both the inflation and the isomorphism bisets have trivial fiber homomorphism, we get $\psi = \tilde\phi$, proving
the equality.

Once more the term $\tw_G^{\tilde\phi}$ may be decomposed further but we do not consider it. For the other term in Equation \ref{eqn:finite1st}, we prove a more general criterion for an idempotent 
$e_{(K, \kappa)}^G$ to be reduced. 

\begin{theorem}\label{thm:redCrit} Let $G$ and $A$ be finite cyclic $p$-groups. Then 
\begin{enumerate}
\item[(i)] If $|A|\ge |G|$, then $e_{(K, \kappa)}\notin I_G$ if and only if $(K, \kappa) = (1,1)$.
\item[(ii)] If $|A|< |G|$, then $e_{(K, \kappa)}\notin I_G$ if and only if $\kappa$ is faithful. 
\end{enumerate}
\end{theorem}
\begin{proof}
If $|A|\ge |G|$, then Equation \ref{thm1:decomposition} is applicable and hence any $A$-fibered $(G,G)$-biset 
can be decomposed into a product of basic fibered bisets. It is clear from the decomposition that the only reduced 
elements are products of twist and isogation. It is also clear that the only element $e_{(K, \kappa)}$ which can be written in this form is $e_{(1,1)}$.

For the second case note that by \cite[Proposition 8.6]{bolcay}, if $e_{(K, \kappa)}\notin I_G$, then $\kappa$
is faithful. Conversely suppose $\kappa$ is faithful. Assume, for a contradiction, that $E_{(K, \kappa)}^G$ is not reduced, and write  
\[
E_{(K, \kappa)}^G = \Big(\frac{G\times H}{U, \phi}\Big)\otimes_{AH} \Big(\frac{H\times G}{V, \psi}\Big)
\]
for a group $H$ of order smaller than $|G|$ and for pairs $(U, \phi)\in \mathcal M_{G\times H}(A)$ and 
$(V, \psi)\in \mathcal M_{H\times G}(A)$. We further let $g$ be a generator of the cyclic group $G$. Then since $(g, g)\in \Delta_K(G)$, there
exists $h\in H$ such that $(g, h)\in U$ and $(h, g)\in V$. Notice that since $|H| < |G|$, the order of $h$ is less than the order of $g$. Now 
consider the subgroup $U'$ of $U$ generated by $(g, h)$. Clearly, $U'$ is isomorphic with $G$ and $k_1(U')$ is isomorphic to the cyclic 
group of order $|G|/|H|$. Also
\[
k_1(U') \le k_1(U) \le k_1(U\ast V) = K.
\]
Here the second inequality follows from \cite[Proposition 2.6 (a)]{bolcay}.  Also by \cite[Proposition 2.6 (b)]{bolcay}, the restriction of $\phi$
to the subgroup $k_1(U')\times 1$ coincides with the restriction of $\kappa \times 1$. In particular, since $\kappa$ is faithful, the 
homomorphism $\phi |_{k_1(U')\times 1}$ is faithful. But clearly $\phi|_{U'}$ extends $\phi |_{k_1(U')\times 1}$ and since $U'$ is cyclic, the
extension must be faithful. This is a contradiction because we assumed that the order of $G$ is strictly greater than $|A|$. Hence
$E_{(K, \kappa)}^G$ must be reduced, as required.
\end{proof}

Note that when $|A| < |G|$ and $\kappa\in K^A$ is not faithful, we have a decomposition
\begin{equation}\label{eqn:non-reduced}
e_{(K, \kappa)} = \infl_{G/\hat K}^G e_{(K/\hat K , \hat\kappa)} \Def^G_{G/\hat K}
\end{equation}
where $\hat K$ is the kernel of $\kappa$ and $\hat\kappa$ is the character induced by $\kappa$. Moreover 
$e_{(K/\hat K , \hat\kappa)}\notin I_{G/\hat K}$.

With the above result, we see that the second term in the factorization (\ref{eqn:finite1st}) is reduced, and hence it does not factor through a
group of smaller order. This completes the proof of the following theorem.
\begin{theorem}\label{thm:factorization}
Let $p$ be a prime number, let $G, H$ and $A$ be finite cyclic $p$-groups with $|G|\ge |H|$ and let $(U, \phi)\in \mathcal M_{G\times 
H}$ such that 
\[
p_1(U) = G,\,\,\, p_2(U) = H,\,\,\, \ker (\phi_1) = 1,\,\,\, \ker (\phi_2) = 1.
\]
Put $K = k_1(U)$ and $L=k_2(U)$.
Then 
\begin{enumerate}
\item[(i)] Suppose $|A|\ge |H|$. Then $\phi_2$ extends to $\phi_H: H\to A$, and let $\psi_G: G\to A$ be as above. Then  the $A$-fibered $(G, H)$-biset $(\frac{G\times H}{U,\phi})$ factors as
\[
\Big( \frac{G\times H}{U, \phi}\Big) \cong \tw_G^{\psi_G}\infl_{G/K}^G\Isom_{G/K, H/L}^{\eta}\Def_{H/L}^H\tw_H^{\phi_H}.
\]
\item[(ii)] Suppose $|A|<|H|$. Let $H = \langle h\rangle, G =\langle g\rangle$ such that $(g,h)\in U$ and 
$\eta: H/k_2(U)\to G/k_1(U)$be the canonical isomorphism induced by $U$. Set $\Delta_\eta(G) = \langle(g,h)\rangle\le G\times H, K_\eta = k_1(\Delta_\eta(G))$. Let $\tilde\eta: H\to G/K_\eta$ be the canonical isomorphism induced by $\Delta_\eta(G)$. Finally let
$\tilde\phi : G\to A, \tilde\phi(g) = \phi(g,h), \,\mathrm{and}\, \phi_L=\phi|_{1\times L}$.
Then the $A$-fibered $(G, H)$-biset $(\frac{G\times H}{U,\phi})$ factors as
\[
\Big( \frac{G\times H}{U, \phi}\Big) \cong \tw^{\tilde\phi}_G \infl_{G/K_\eta}^G \Isom_{G/K_\eta, H}^{\tilde\eta} E_{L, \phi_L}^H.
\]
\end{enumerate}
\end{theorem}

\section{Simple fibered biset functors with cyclic minimal groups}\label{sec:paramsimple}

A classification of simple fibered biset functors is given in \cite[Section 9]{bolcay}. In this paper we only need
a special case of this result when the minimal group is cyclic. In this section we specialize to this case.

Let $S$ be a simple $A$-fibered biset functor and $G$ be a group such that $S(G)
 \not = 0$ but $S(H) = 0$ for any $H$ with $|H| < |G|$. In this case, we call $G$ a minimal group for $S$. 
 As explained in \cite[Section 4]{bouc}, for any finite group $K$, the evaluation $S(K)$ is either zero or a simple module 
 over the algebra $E_K =  \End_{\mathcal C} (K)$. Moreover, for a minimal group $G$, the module $S(G)$
 is a simple module over the essential algebra $\bar E_G$ which is defined as follows. Let $I_G$ be the ideal 
 generated by all elements in $E_G$ which factor through a group of smaller order. Then $\bar E_G =
 E_G / I_G$. A detailed structure of the essential algebra is described in \cite{bolcay}. 
 
 Conversely given a simple module $V$ over the essential algebra $\bar E_G$, there is a unique simple $A$-
 fibered biset functor $S$ with minimal group $G$ and $S(G) \cong V$. However, for the simple functor $S$, 
 the minimal group $G$ is not always unique up to isomorphism. By \cite[Theorem 9.2]{bolcay}, the minimal 
 groups for a given simple fibered biset functor are parameterized by linkage classes of certain triplets. 

Therefore to get a parametrization of simple fibered biset functors with cyclic minimal groups, we need to 
determine simple modules over the essential algebra and the linkage classes for cyclic groups. The following result recalls the structure of the essential algebra for two particular cases. 
 
 \begin{theorem}\label{thm:essKnown}
Let $R = k$ be a field. Then
 \begin{enumerate}
 \item[(i)] {\rm{(Bouc, \cite{bouc})}} If $|G|$-torsion of $A$ is trivial, then $\bar E_G \cong k\out(G)$.  
 \item[(ii)] {\rm{(Coşkun - Yılmaz, \cite{CY})}} If $G$ is abelian and $A$ has a non-trivial element of order $|G|$, then $\bar E_G \cong k[G^*\ltimes \Aut(G)].$ 
 \end{enumerate}
 \end{theorem}


\noindent Suppose $G$ is cyclic. By Lemma \ref{lem:direct-product}, for any abelian group $A$, there is an isomorphism of algebras
\[
\bar E_G \cong \prod_{p_i} \bar E_{P_i}.
\]
Here $p_i$ runs over all distinct prime divisors of $|G|$ and for each $i$, we denote by $P_i$ the Sylow 
$p_i$-subgroups of $G$. In particular if $p_i$ does not divide $|A|$, or if the 
$p_i$-part of $|A|$ is greater than that of $|G|$, the structure of $\bar E_{P_i}$ is given by Theorem 
\ref{thm:essKnown}. Thus we need only to discuss the case where $p_i$-part of $|A|$ is non-trivial and less
than that of $|G|$. Hence let $p$ be such a prime and assume that $G$ and $A$ are finite $p$-groups with 
$|A|<|G|$. 

By \cite[Corollary 8.5]{bolcay}, simple modules over the algebra $\bar E_G$ are parameterized by simple 
modules over certain group algebras. To describe this parametrization, we introduce further notation. Let 
$\mathcal R_G$ be the subset of $\mathcal M_G$ consisting of reduced pairs, that is,
\[
\mathcal R_G = \{ (K, \kappa)\in \mathcal M_G : e_{(K, \kappa)}\notin I_G\}.
\]
The triplets $(G, K, \kappa), (G, L, \lambda)$ are said to be \textit{linked}, written $(G, K, \kappa) \sim (G, L, \lambda)$, if
there is an $A$-fibered $(G,G)$-biset $(\frac{G\times G}{U, \phi})$ such that $l(U, \phi) = (G, K, \kappa)$ and $r(U, \phi) = (G, L, \lambda)$.
In this case we say that $(\frac{G\times G}{U, \phi})$ links $(G, K, \kappa)$ to $(G, L, \lambda)$.
This induces an equivalence relation on $\mathcal M_G$. The equivalence class, called the linkage class, containing $(K, \kappa)$ is denoted
by $\{ K,\kappa \}$.  Note that if $(G, K, \kappa) \sim (G, L, \lambda)$, then $|K| = |L|$ as noted in \cite[Definition 5.1]{bolcay}. Since $G$ is 
cyclic, we further have $K = L$. Finally note that if $(G, K, \kappa)\sim (G, K, \lambda)$, then $(K, \kappa)\in \mathcal R_G$ if and only if 
$(K, \lambda)\in \mathcal R_G$, by \cite[Section 8.1]{bolcay}. 

Given $(K, \kappa)\in \mathcal R_G$, we write 
$\Gamma_{(G, K, \kappa)}$ for the set of isomorphism classes of all $A$-fibered $(G,G)$-bisets $(\frac{G\times G}{U, \phi})$ such that 
$l(U, \phi) = (G, K, \kappa) = r(U, \phi)$. By \cite[Section 6.1]{bolcay}, it becomes a group under the tensor 
product, with the identity element $e_{(K, \kappa)}$ and inverses given by opposite fibered bisets. 

With this 
notation, Corollary 8.5 in \cite{bolcay} gives a bijective correspondence between
\begin{enumerate}
\item[(i)] isomorphism classes of simple $\bar E_G$-modules and
\item[(ii)] triples $(K,\kappa, [V])$ where $(K, \kappa)$ runs over $\mathcal R_G$, up to linkage, and
$[V]$ runs over isomorphism classes of simple $k\Gamma_{(G,K,\kappa)}$-modules.
\end{enumerate}
The correspondence is given by $(K, \kappa, [V])\mapsto \bar E_G \bar f_{(K, \kappa)}\otimes_{k\Gamma_{(G, K, \kappa)}}V$ where $ \bar f_{(K, \kappa)}$ is the image of  $f_{(K, \kappa)}$ under the canonical projection onto $ \bar E_G$.

 
It turns out that when $G$ is cyclic, there is only one linkage class of reduced pairs once we fix the subgroup 
$K$. More precisely we prove the following proposition.
\begin{prop}\label{pro:LinkageCyclic}
Let $p$ be a prime number and $G$ and $A$ be finite cyclic $p$-groups with $|A| < |G|$. Given $K\le G$, the 
triplets $(G, K, \kappa)$ and $(G, K, \lambda)$ are linked for any faithful $\kappa, \lambda\in K^A$.
\end{prop}
\begin{proof}
For simplicity, we write $\kappa \sim\lambda$ instead of $(G, K, \kappa)\sim (G, K, \lambda)$. Since $\kappa$
and $\lambda$ are faithful, the pairs $(K, \kappa)$ and $(K, \lambda)$ are both reduced by Theorem \ref{thm:redCrit}. Thus if $\kappa\sim
\lambda$, then  by Theorem \ref{thm:factorization}, any link $(\frac{G\times G}{U, \phi})$ between $\kappa$ and $\lambda$ is of the form
\[
\tw_G^{\psi} \Isom_{G,G}^\alpha E_{(K, \kappa)}^G
\]
for some $\psi\in G^A$ and $\alpha\in \Aut(G)$. In particular, the linkage class of $\kappa$ is determined by the left invariants of products of the above type. To determine these invariants, let $\kappa\in K^A, \psi\in G^A$ and $\alpha\in
\Aut(G)$. We consider the product 
\[
 \tw_G^{\psi} \Isom_{G,G}^\alpha E_{(K, \kappa)}^G = \Big(\frac{G\times G}{\Delta(G),\Delta_\psi} \Big)
\Big(\frac{G\times G}{{}_{\alpha}\Delta(G), 1} \Big) \Big(\frac{G\times G}{\Delta_K(G), \phi_\kappa} \Big)
\]
 where ${}_{\alpha}\Delta(G):=\{(\alpha(g),g)\, | \, g\in G\}$. By the Mackey product formula for fibered 
bisets, the above product is a transitive fibered biset with stabilizing pair $(U, \phi)$ where
\[
U = \Delta(G)\ast {}_{\alpha}\Delta(G)\ast \Delta_K(G) = \{(g_1, g_2): \alpha(g_1^{-1})K = g_2K\}
\]
and $\phi(g_1, g_2) = \psi(g_1)\kappa(g_2^{-1}\alpha(g_1^{-1}))$ for $(g_1, g_2)\in U$. In particular we have 
\[
p_1(U) = G = p_2(U), k_1(U) = K = k_2(U).
\]
Here we note that since $G$ is cyclic, we have $\alpha(K) = K$. Moreover for $g\in K$, we have
\[
\phi(g, 1) = \psi(g)\kappa(\alpha(g^{-1})) = (\psi\cdot {}^{\alpha^{-1}}\kappa)(g)
\,\,\,
\mbox{\rm and}\,\,\,
\phi(1, g) = \psi(1)\kappa(g^{-1}) = \kappa^{-1}(g).
\]
Therefore we obtain
\[
l(U, \phi) = (G, K, \psi\cdot {}^{\alpha^{-1}}\kappa)
\,\,\,
\mbox{\rm and}\,\,\,
r(U, \phi) = (G, K, \kappa).
\]
In particular we see that the pairs $(K, \kappa)$ and $(K, \lambda)$ are linked if and only if 
there exists $\phi\in G^A$ and $\alpha\in \Aut(G)$ such that $\lambda =\phi\cdot{}^{\alpha}\kappa$. 
To see that $\kappa\sim\lambda$ for any faithful $\kappa$ and $\lambda$, note that,
since the characters $\kappa$ and $\lambda$ are faithful, their extensions $\chi_1$ and $\chi_2$ to
$G$ are also faithful. Hence there exists an automorphism $\sigma$ of $G$ such that $\chi_2 = 
{}^\sigma\chi_1$ and hence $\lambda = {}^\sigma\kappa$. 
\end{proof}

Therefore when $G$ and $A$ are finite cyclic $p$-groups with $|A| < |G|$, the set of linkage classes on 
$\mathcal R_G$ is in one-to-one correspondence with the set of subgroups of $G$ of order at most $|A|$. 
Hence in this case, there is a bijective correspondence between 
\begin{enumerate}
\item[(i)] the isomorphism classes of  simple $\bar E_G$-modules
\item[(ii)] the pairs $(K, [V])$ where $K\le G$ runs over subgroups of order at most $|A|$ and $[V]$ runs over isomorphism classes of simple $k\Gamma_{G, K, \kappa}$-modules for a fixed faithful character $\kappa$ of
$K$.
\end{enumerate}

Now let $G$ be a cyclic $p$-group and $A$ be abelian. Given a simple $A$-fibered biset functor $S$ with 
minimal group $G$ and $S(G) = V$. If $|G|$-torsion of $A$ is trivial, then $A$-fibered $(G,G)$-bisets can be 
identified with $(G,G)$-bisets and hence by \cite[Theorem 4.3.10]{bouc}, $G$ is the 
unique minimal group for $S$, up to isomorphism and the only reduced pair for $G$ is $(1, 1)$. Hence we 
have $S = S_{G, V} = S_{G, 1, 1, [V]}^A$.

Otherwise if $A$ has an element of order $|G|$, then by Equation \ref{thm1:decomposition}, $G$ is unique up 
to isomorphism, again there is only one reduced pair for $G$. Thus we have $S = S_{G, 1, 1, [V]}^A$. 

Finally in the remaining case, suppose $S(G)$ corresponds to the triple $(K, \kappa, [V])$. By 
\cite[Proposition 9.7]{bolcay}, any other minimal group $H$ for $S$ must be 
abelian of order $|G|$ and there must exist $(L, \lambda)\in \mathcal M_H$ such that $(G, K, \kappa)$ is linked to $(H, L, \lambda)$. As we shall see at the end of Section \ref{sec:mainR}, for simple subfunctors of 
the functor of complex characters, $H$ must be isomorphic to $G$.

\section{Actions on characters}\label{sec:main}

We denote by Irr$(G)$ the set of irreducible complex characters of the group $G$ and by  
$$R_{\mathbb{C}}(G) = \bigoplus_{\chi\in\mbox{\rm Irr}(G)} \mathbb Z \cdot \chi$$ the ring of complex characters of $G$. As usual, we identify it with the Grothendieck ring of the category of finite dimensional 
$\mathbb C G$-modules, and identify the set Irr$(G)$ by a complete set of isomorphism classes of simple 
$\mathbb C G$-modules.  Given a $\C G$-module $M$, we denote its image in $R_{\mathbb C}(G)$ by 
$\chi_M$. 
The functor sending $G$ to $R_\CC(G)$ is a $\C^\times$-fibered biset functor as described in \cite{bolcay}. In the
 same way it has a structure of an $A$-fibered biset functor for any subgroup $A$ of $\Ctimes$. We recall this 
 structure.

Let $A$ be a subgroup of $\mathbb C^\times$. As in the previous section, let $B^A$ denote the $A$-fibered biset functor of Burnside groups, sending any group $G$ to the Burnside group $B^A(G)$. 
Given a transitive $A$-fibered $G$-set $X = [U,\phi]_G$, we can construct the monomial $\C G$-module $\C X$ with 
monomial basis $X$, that is, the $\C$-vector space  $\mathbb CX$ with basis the $A$-orbits $X/A$ of
$X$ and the $\C G$-action inherited from the $G$-action on $X/A$. It is easy to show that 
\[
\C X\cong \ind_U^G \C_\phi
\]
as $\C G$-modules, where $\C_\phi$ denotes the 1-dimensional representation of $U$ with character $\phi$. 
Then the well-known linearization map 
\[
\lin_G: B^A(G) \to R_{\mathbb C}(G)
\]  
is defined as the linear extension of this correspondence. Similarly, if $Y$ is an $A$-fibered $(H,G)$-biset, then the linearization of $Y$ can be regarded as a monomial $(\C H,\C G)$-bimodule and hence induces a group homomorphism 
\[
R_\mathbb C(Y): R_\mathbb C(G)\to R_\mathbb C(H)
\]
given by $R_\C(Y)(\chi_M) = \chi_{\C Y\otimes_{\C G} M}$. For simplicity, we denote this homomorphism by 
$Y_H^G$. It is shown in \cite[Section 11.4]{bolcay} that with this induced action of fibered bisets, the
functor $R_\mathbb C$ becomes an $A$-fibered biset functor. Note that when $A$ is trivial, the above
definition gives a biset functor structure on the functor $R_C$ which also becomes a Green biset functor. The following theorem summarizes the results describing the biset functor structure of $\crc$. We refer to \cite{TW} for the 
Mackey functor structure of it.
\begin{theorem}
\begin{enumerate}
\item[\rm{(i)}] \rm{(Bouc,  {\rm\cite{bouc}})} The biset functor $\crc$ is semisimple and there is an isomorphism of biset functors
\[
\mathbb C R_\mathbb C \cong \bigoplus_{(m,\zeta)} S_{\mathbb Z/m\mathbb Z, \C_\zeta}
\]
where the sum is over all pairs $(m, \zeta)$ with $m\in \mathbb N$ and $\zeta$ runs over all primitive characters of $(\mathbb Z/m\mathbb Z)^\times$.
\item[\rm{(ii)}] \rm{(Boltje - Co\c{s}kun, \rm\cite{bolcay})} The $\mathbb C^\times$-fibered biset functor $\crc$ is simple and there is an isomorphism of $\mathbb C^\times$-fibered biset functors
\[
\mathbb C R_\mathbb C \cong S_{1,1,1,1}
\]
where the right hand side is the unique simple $\mathbb C^\times$-fibered biset functor with minimal group $1$.
\item[\rm{(iii)}] \rm{(Romero, \rm\cite{romero})} The Green biset functor $\mathbb C R_\mathbb C$ is simple.
\end{enumerate}
\end{theorem}
Note that Bouc's decomposition can be thought as a decomposition of $1$-fibered biset functors, and hence parts (i) and (ii) of the above theorem cover two extreme cases where the fiber group is the smallest and the 
largest. For the rest of the paper, we want to determine the structure of $\mathbb C R_\mathbb C$ as an 
$A$-fibered biset functor for other subgroups of $\mathbb C^\times$. 
For this aim, we need the following description of action of fibered bisets on
characters. Note that this is a special case of \cite[Lemma 7.1.3]{bouc}.
\begin{lemma}\label{lem:GenAction}
Let $M$ be a $\mathbb CG$-module with character $\chi$ and $Y$ be a transitive $A$-fibered $(H, G)$-biset with stabilizing pair 
$(V, \psi)$. Then for any $h\in H$, the value at $h$ of the character $Y_H^G \chi$ of the $\mathbb CH$-module $\mathbb CY\otimes_{\mathbb CG} M$ is given by
\[
(Y_H^G \chi)(h) = \frac{1}{|V|}\sum_{\substack{x \in H, g \in G \\ (h^x,g) \in V}}\psi(h^x,g)\chi(g).
\]
\end{lemma}

\begin{proof} We apply Bouc's formula \cite[Lemma 7.1.3]{bouc} for characters of tensor product of bimodules  to obtain
\[
(Y_H^G\chi)(h) = \frac{1}{|G|}\sum_{g \in G}\theta(h,g)\chi(g)
\]
where $\theta$ is the character of the monomial $(\C H, \C G)$-bimodule $\C Y$. By definition it is the function 
sending any $(h,g) \in H \times G$ to the trace of the endomorphism $y \mapsto (h,g)y$ of $\C Y$. Since $Y$ is an $A$-fibered
$(H,G)$-biset and $A \le \Ctimes$, any set $[Y/A]$ of representatives of $A$-orbits $Y/A$ of $Y$ is a basis of 
$\C Y$. Precisely
\[
\theta(h,g)=\sum_{y \in [Y/A] \atop (h,g)[y]=[y]}\psi_{y}(h,g)
\]
where $\psi_{y}(h,g) \in A$ is determined by the equation $(h,g)y=\psi_{y}(h,g)y$ and we write $[y]$ for the $A$-orbit containing $y$. On the other hand, since $Y$ is transitive, for any $[y],[y']\in Y/A$, there exists some 
$(a,b)V \in (H \times G) /V$ such that $(a,b)[y]=[y']$. Also, if $(h,g)$ stabilizes $[y]$, then $(h,g)^{(a,b)}$ stabilizes 
$[y']$. Thus,
\begin{eqnarray*}
(Y_H^G \chi)(h) &=& \frac{1}{|G|}\sum_{g \in G}\theta(h,g)\chi(g) 
= \frac{1}{|G|}\sum_{g \in G}\sum_{(a,b)V \in (H \times G) /V \atop (h^a, g^b) \in V}\psi(h^a,g^b)\chi(g).
\end{eqnarray*}
Replacing $g$ with ${}^bg$ we get
\begin{eqnarray*}
Y_H^G (\chi)(h) &=& \frac{1}{|G|}\sum_{^bg \in G}\sum_{(a,b)V \in (H \times G) /V \atop (h^a, g) \in V}\psi(h^a,g)\chi(g).
\end{eqnarray*}
If $(a,b)V \in (H \times G) /V$, then for each $(u,v) \in V$ and for any $[y] \in Y/A$,
\[
(h^{au},g^v)[y]=(h^a,g)^v[y]=(h^a,g)[y].
\]
Hence
\begin{eqnarray*}
(Y_H^G \chi)(h) &=& \frac{1}{|G| |V|}\sum_{^bg \in G}\sum_{(a,b) \in (H \times G) \atop (h^a, g) \in V}\psi(h^a,g)\chi(g) 
= \frac{1}{|G| |V|}\sum_{g \in G}\sum_{(a,b) \in H \times G \atop (h^a, g) \in V}\psi(h^a,g)\chi(g) \\
&=& \frac{1}{|G| |V|}\sum_{g \in G}\sum_{a \in H}\sum_{b \in G : \atop (h^a,g) \in V}\psi(h^a,g)\chi(g) 
= \frac{1}{|V|}\sum_{a \in H, g \in G \atop (h^a,g) \in V}\psi(h^a,g)\chi(g).
\end{eqnarray*}
This completes the proof of the lemma.
\end{proof}

We also need explicit descriptions of actions of certain fibered bisets. Note that if $Y$ is one of induction, 
restriction, inflation, deflation or isogation bisets, then the above formula becomes the corresponding well-known 
map from character theory. We also have the following maps.

\begin{coro}\label{cor:actionTw}
Let $\chi \in R_\mathbb C(G)$. Then
\begin{enumerate}
\item[(i)] For any $\phi\in G^A$, we have
\[
\tw_G^\phi \chi = \phi\cdot\chi.
\]
\item[(ii)] Suppose $G$ is abelian and let $(K, \kappa)\in \mathcal M_G$. Then 
\[
E_{(K, \kappa)}^G \chi = \sum_{\phi\in\Irr(G): \phi |_K = \kappa} <\chi, \phi> \phi.
\]
In other words, the action of $E_{(K, \kappa)}^G$ on a $\mathbb CG$-module $M$ is given by projection onto the 
submodule of $M$ generated by all extensions of $\kappa$ to $G$.
\end{enumerate}
\end{coro}
\begin{proof}
For the first part, recall that the stabilizing pair of $\tw_G^\phi$ is $(\Delta(G), \Delta_\phi)$. Thus by Lemma \ref{lem:GenAction}, for any
$g\in G$, we get
\begin{eqnarray*}
\tw_G^\phi(\chi)(g) &=& \frac{1}{|\Delta(G)|}\sum_{x \in G, g' \in G \atop (g^x,g') \in \Delta(G)}\Delta_\phi((g^x,g'))\chi(g') 
= \frac{1}{|G|}\sum_{x \in G}\Delta_\phi((g^x,g^x))\chi(g^x) \\
&=& \frac{1}{|G|}\sum_{x \in G}\phi(g)\chi(g) =\phi(g)\chi(g) = (\phi\chi)(g).
\end{eqnarray*}
For the second part, suppose $G$ is abelian. The stabilizing pair of $E_{(K,\kappa)}^G$ is $(\Delta_K(G), \phi_\kappa)$ where 
$\phi_\kappa(g_1, g_2) = \kappa(g_1{g_2}^{-1})$. To simplify calculations, 
we suppose $\chi$ is an irreducible character. The general case follows from the linearity of the action.

Once more, by Lemma \ref{lem:GenAction}, for any $h\in G$, we have
\begin{eqnarray*}
 (E_{(K, \kappa)}^G\chi)(h) &=& \frac{1}{|G| |K|}\sum_{\substack{x \in G, g \in G \\ (h^x,g) \in \Delta_K(G)}}\phi_\kappa(h^x,g)\chi(g)
=\frac{1}{|K|}\sum_{\substack{g \in G: (h,g) \in \Delta_K(G)}}\kappa(hg^{-1})\chi(g)\\
& =& \frac{1}{|K|}\sum_{\substack{k \in K }}\kappa(k^{-1})\chi(hk) 
= \chi(h)\Big(\frac{1}{|K|}\sum_{\substack{k \in K }}\kappa(k^{-1})\chi(k)\Big).
\end{eqnarray*}
Here note that the last equality holds since $G$ is abelian and hence $\chi$ is a homomorphism. Notice that the second term in the last
row above is the inner product of the characters $\kappa$ and $\chi |_K$ in the space of class functions on $K$, and hence is equal to 1
if $\kappa = \chi |_K$ and zero otherwise. Hence the result follows. 
\end{proof}

We also need action of the idempotent $f_{(K,\kappa)}^G$ when $G$ is cyclic. In this case, the description of the idempotent is simpler.
\begin{coro}\label{lem:actionIpots}
Let $p$ be a prime and $G$ and $A$ be finite cyclic $p$-groups with $|A| < |G|$. Also let $(K,\kappa)\in \mathcal R_G$. Then
\begin{enumerate}
\item[(i)] If $|K| = |A|$, then $$f_{(K, \kappa)}^G = e_{(K, \kappa)}^G.$$
In particular, $f_{(K, \kappa)}^G$ annihilates all irreducible characters of $G$ whose restriction to $K$ does not coincide with $\kappa$.
\item[(ii)] If $|K| < |A|$, let $L\le G$ be the unique subgroup of $G$ of order $p\cdot |K|$. Then 
\[
f_{(K, \kappa)}^G = e_{(K, \kappa)}^G - \sum_{\lambda\in L^A: \lambda |_K = \kappa} e_{(L, \lambda)}^G.
\] 
Moreover, in this case, $f_{(K, \kappa)}^G$ annihilates all characters of $G$. 
\end{enumerate}
\end{coro}
\begin{proof}
For both parts, the second claims follows form the above corollary once we prove the first parts. Thus we only prove the first claims.
Recall that, by definition, for any pair $(K, \kappa)\in\mathcal M_G$, we have
\[
f_{(K,\kappa)}^G = \sum_{(K, \kappa)\le (L, \lambda)\in \mathcal M_G} \mu^\triangleleft_{K, L}e_{(L, \lambda)}^G.
\]
Since $G$ is cyclic, the Mobius function $\mu^\triangleleft_{K, L}$ coincides with the number theoretic Mobius function $\mu(|L:K|)$ and hence by \cite[Proposition 10.1.10]{Co} we have
\[
\mu^\triangleleft_{K, L} =   \left\{
\begin{array}{ll}
      1 & \mbox{\rm if}\,\,\, K = L, \\
      -1 & \mathrm{if}\,\,\, |L:K| = p, \\
      0 & \mathrm{otherwise.}\\
\end{array} 
\right. \]
Now given $(K, \kappa)\in \mathcal R_G$, there is a unique subgroup $L\le G$ such that $|L:K| = p$. Also if $(K, \kappa) \le (L, \lambda)$,
then $\lambda$ is an extension of $\kappa$. Since $\kappa$ is faithful, an extension exists if and only if 
$|L| \le |A|$.
\end{proof}

Finally we discuss actions of isogations and twists on primitive characters. Again, by Lemma \ref{lem:actionCoprime}, 
it is sufficient to consider the case of $p$-groups.
Let $p$ be a prime, $A$ be cyclic of order $p^l$ and $m = p^k$ for some $k\in \mathbb N$ with $l<k$. We write $G = \mathbb Z/m\mathbb Z$ and identify $A$ with the unique subgroup of $G$ of order $p^l$. Also let $\nu$ be
a primitive character modulo $m$ and write $\tilde\nu$ for the canonical extension of $\nu$ to $G$. 

Recall that a primitive character $\nu$ modulo $m$ is a multiplicative homomorphism $\nu : (\zm)^\times \to \Ctimes$ that is not induced from any character of smaller modulus. With this notation, $\tilde\nu : \zm \to \C$ is given by $\tilde\nu(x)=\nu(x)$ if $ x\in (\zm)^\times$ and zero otherwise.

The set $G^* = \Irr(G)$ is also cyclic, isomorphic to $G$. We write $\chi$ for the generator of $G^*$
for which $\chi(1) = e^{\frac{2\pi i}{m}}$. With this choice, we get
\[
G^* =\{1, \chi, \chi^2,\ldots,\chi^{m-1}\}
\]
and $\chi^j(1) = \chi(j)$ for any $j = 0, 1, \ldots, m-1$. Hence by \cite[Corollary 2.1.42]{Co}, the generalized character $\tilde\nu$ can be expressed as 
\[
\tilde\nu = \tau(\tilde\nu)\sum_{(j, p) = 1}\overline{\tilde\nu(j)}\chi^j
\]
where the sum is over all indexes $j = 0, 1,\ldots, m-1$ such that $(j, p) = 1$ and where $$\tau(\tilde\nu) = 
<\tilde\nu, \chi>$$ is the coefficient of the irreducible character $\chi$ in the generalized character 
$\tilde\nu$. 

Recall that by \cite[Section 4.3]{bolcay}, the idempotents $f_{\{K,\kappa\}}$ sum up to the identity
endomorphism of $G$ as $(K, \kappa)$ runs over the linkage classes. Furthermore, by Corollary \ref{lem:actionIpots}, 
the idempotents $f_{\{K,\kappa\}}$ annihilates any character of $G$ if $(K, \kappa)\in \mathcal R_G$ and $K < A$. Moreover, if $(K, \kappa)$ is not reduced, then each summand of $f_{(K, \kappa)}$ is also non-reduced, and by 
Equation \ref{eqn:non-reduced}, it annihilates all primitive characters. Finally, by Proposition \ref{pro:LinkageCyclic}, all faithful $\alpha\in A^*$
are linked. Therefore we get
\[
\tilde\nu = 1_G\cdot \tilde\nu = f_{\{A,\alpha\}}\tilde\nu = \sum_{\alpha\in A^*: \alpha\, \mbox{\rm\scriptsize faithful}} f_{(A, \alpha)}\tilde\nu.
\]
For simplicity, we put $\tilde\nu_\alpha = f_{(A,\alpha)}\tilde\nu$ and get
\[
\tilde\nu = \sum_{\alpha\in A^*: \alpha\, \mbox{\rm\scriptsize faithful}}\tilde\nu_\alpha.
\]
Notice that by Corollary \ref{lem:actionIpots}, the summand $\tilde\nu_\alpha$ is the projection of $\tilde\nu$
to the summand whose restriction to $A$ is a multiple of $\alpha$, that is,
\[
\tilde\nu_\alpha = \tau(\tilde\nu)\sum_{j: \chi^j|_A = \alpha} \overline{\tilde\nu(j)}\chi^j
\]
where the sum is over all distinct extensions of $\alpha$ to $G$. This set can be described more explicitly. Let $j_\alpha$ be the smallest natural number such that $\chi^{j_\alpha}$ extends $\alpha$. Then any other
extension is a product of $\chi^{j_\alpha}$ and an irreducible character of $G$ inflated from the quotient group
$G/A$. In particular, if $\chi^j$ is another extension of $\alpha$, then $j\cong j_\alpha$ (mod $|A|$). Hence we have
\[
\{\chi^j \in G^*: \chi^j|_A = \alpha \} = \{\chi^{j_\alpha + k\cdot|A|}: k = 0, 1, \ldots, |G:A| -1\}.
\]
Therefore we have
\begin{equation}\label{eqn:nuAlpha}
\tilde\nu_\alpha = \tau(\tilde\nu)\sum_{k=0}^{|G:A|-1} \overline{\tilde\nu(j_\alpha + k|A|)}\chi^{j_\alpha + k|A|}
= \chi^{j_\alpha}\cdot\big(\tau(\tilde\nu)\sum_{k=0}^{|G:A|-1} \overline{\tilde\nu(j_\alpha + k|A|)}\chi^{k|A|}\big).
\end{equation}

Now we determine the actions of fibered bisets on these summands. First, for any $\sigma\in \Aut(G)$, the 
action of the isogation biset $\Isom_G^\sigma$ on $\tilde\nu$ follows from \cite[Theorem 7.3.4]{bouc}. We 
identify the groups $G^\times = \Aut(G)$ via $\sigma \leftrightarrow \sigma(1)$.
\begin{coro}[Bouc]\label{cor:actionIso}
With the above notation, given $\sigma\in\Aut(G)$, we have
\[
\Isom_G^\sigma\tilde\nu_\alpha = \overline{\tilde\nu(\sigma(1))}\tilde\nu_{{}^\sigma\alpha}
\]
where ${}^\sigma\alpha (x) = \alpha(\sigma(x))$ for any $x\in A$.
\end{coro}

Next we determine the action of twist bisets $\tw_G^\phi$ for $\phi\in G^A$ on $\tilde\nu_\alpha$ for a fixed 
fatihful $\alpha\in A^*$. Since $G^A< G^*$, for any $\phi\in G^A$, we have $\phi = \chi^s$ for some 
$s\in \mathbb N$ such that $|G:A|$ divides $s$. 
\begin{lemma}\label{lem:actionTw}
Let $\chi^s\in G^A$ be a character of $G$ and let $\alpha\in A^*$ be faithful. Then
\[
\tw_G^{\chi^s}\tilde\nu_\alpha = \nu(x) \tilde\nu_\beta
\]
where $\beta = \chi^s \cdot\alpha$ and $x\in G^\times$ is given by $j_\beta^{-1}(j_\beta - s).$
\end{lemma}
\begin{proof}
By Corollary \ref{cor:actionTw}, we have
\[
\tw_G^{\chi^s}\tilde\nu_\alpha = \chi^s \tilde\nu_\alpha.
\]
Hence given $\chi^j\in G^*$, we have
\[
<\tw_G^{\chi^s}\tilde\nu_\alpha, \chi^j> = <\tilde\nu_\alpha, \chi^{j-s}>.
\]
In particular, the irreducible character $\chi^j$ appears in $\tw_G^{\chi^s}\tilde\nu_\alpha$ with a non-zero 
coefficient if and only if $\chi^{j-s}$ appears in $\tilde\nu_\alpha$ with a non-zero coefficient. This is equivalent
to say that $\chi^{j-s}|_A = \alpha$. Hence we have
\[
\tw_G^{\chi^s} \tilde\nu_\alpha = \tau(\tilde\nu) \sum_{\chi^j|_A = \beta} \overline{\tilde\nu(j-s)}\chi^j
\]
where we put $\beta = \chi^s|_A\cdot \alpha$. As above, we can rewrite this sum in the following way.
\begin{equation}\label{eqn:loc1}
\tw_G^{\chi^s} \tilde\nu_\alpha =\chi^{j_\beta}\cdot\big( \tau(\tilde\nu) \sum_{k=0}^{|G:A|-1} 
\overline{\tilde\nu(j_\beta-s + k|A|)}\chi^{k|A|}\big).
\end{equation}
Now since $s$ is a multiple of $|G:A|$, both $j_\beta$ and $j_\beta - s$ are coprime to $p$. Therefore there is
an automorphism $x\in G^\times$ such that $x j_\beta \equiv j_\beta - s$ mod $|G|$. Then since $s$ is a multiple of $|G:A|$, we also get
\[
xj_\beta + k|A| \equiv x(j_\beta + k|A|) \,\, (\mathrm{mod} |G|)
\]
for all $k = 0,1, \ldots, |G:A|-1$. 
With this congruence, and
the fact that $\tilde\nu$ is periodic with period $|G|$, we can rewrite Equation \ref{eqn:loc1} as
\begin{equation}\label{eqn:loc2}
\tw_G^{\chi^s} \tilde\nu_\alpha =\chi^{j_\beta}\cdot\big( \tau(\tilde\nu) \sum_{k=0}^{|G:A|-1} 
\overline{\tilde\nu(x(j_\beta + k|A|))}\chi^{k|A|}\big).
\end{equation}
Furthermore, since both $x$ and $j_\beta + k|A|, k= 0, 1, \ldots, |G:A|-1$ are in $G^\times$, we have 
$\tilde\nu(x(j_\beta + k|A|)) = \nu(x(j_\beta + k|A|)) = \nu(x) \nu(j_\beta + k|A|)$ and hence Equation \ref{eqn:loc2} becomes
\begin{eqnarray*}
\tw_G^{\chi^s} \tilde\nu_\alpha &=& \nu(x)\chi^{j_\beta}\big( \tau(\tilde\nu) \sum_{k=0}^{|G:A|-1} 
\overline{\tilde\nu(j_\beta + k|A|)}\chi^{k|A|}\big)\\
&=&\nu(x) \tilde\nu_\beta.
\end{eqnarray*}
\end{proof}

The following result describes action of fibered bisets on characters of direct products of groups of coprime orders.
We left the straightforward verification to the reader. Recall that by Theorem
10.33 of \cite{CR}, given groups $G$ and $H$, we have
\[
\crc (G\times H) \cong \crc(G)\otimes_\mathbb Z \crc(H).
\]
In particular, if $G = G_1\times G_2$ and $H = H_1\times H_2$, we can write
\[
\crc (G\times H) \cong \crc(G_1\times H_1)\otimes_\mathbb Z \crc(G_2\times H_2).
\]

\begin{lemma}\label{lem:actionCoprime}
Let $G = G_1\times G_2$ and $H = H_1\times H_2$ be finite groups such that $(|G_1| |H_1|, |G_2| |H_2|) = 1$ and let $X$ be an $A$-fibered $(G,H)$-biset. Also let $\chi\in \crc(H)$. Then
\[
X_H^G \chi = Y_{H_1}^{G_1}(\chi_1) \times Z_{H_2}^{G_2}(\chi_2)
\]
where $X = Y\times Z$ is factored according to Lemma \ref{lem:direct-product} and $\chi = \chi_1\times
\chi_2$ is factored according to the above isomorphism.
\end{lemma}

\section{Functor of complex characters}\label{sec:mainR}
In this section, we prove our main theorems. We distinguish between two cases. The first case is where the fiber group is large enough so that  Equation 1 can be applied. In the second case, we consider a finite fiber group.  It turns out that, in both cases, the functor $\crc$ is semisimple. We also determine its simple summands. We explain 
our approach briefly. Let $A$ be a group as described above. Also let $F$ be an $A$-fibered biset subfunctor 
of $\crc$. Then by restriction, $F$ is also a biset functor and hence it is a direct sum of simple biset functors
of the form $S_{\mathbb Z/m\mathbb Z, \mathbb C_\zeta}$ for some set of pairs $(m, \zeta)$ with $m\in\mathbb
N$ and $\zeta$ is a primitive character of $(\mathbb Z/m\mathbb Z)^\times$. Thus to determine simple $A$-fibered subfunctors of $\crc$, one needs to put an appropriate equivalence relation on the set $\Upsilon$ of
all such pairs.

\subsection{Large enough fiber group.}\label{sec:mainLarge}
Throughout this section, let $\pi$ be a set of prime numbers and $A = \pi^\infty$ be the group of all $p$-power roots of unity for all primes $p\in\pi$, that is,
$$\pi^\infty = \prod_{p\in \pi}{\lim_{\stackrel{\longleftarrow}{n\in\mathbb N}}}(\mathbb Z/p^n\mathbb Z)^\times.$$ 

Let $\Upsilon$ be the set of all pairs $(m, \zeta)$ defined above. We define a relation on $\Upsilon$ as follows. Two pairs $(m,\zeta), (n,\nu)\in \Upsilon$ are said 
to be \textit{$\pi$-equivalent}, written $(m,\zeta)\equiv (n,\nu)$ if the $\pi^\prime$-parts $m_{\pi^\prime}$ and 
$n_{\pi^\prime}$ of $m$ and $n$ are equal and after identifying the groups $\mathbb Z/m_{\pi^\prime}\mathbb Z \cong 
\mathbb Z/n_{\pi^\prime}\mathbb Z$, we have
\[
\mathbb C_{\zeta_{\pi^\prime}} \cong \mathbb C_{\nu_{\pi^\prime}}.
\]
The last condition means that the $\pi^\prime$-parts ${\zeta_{\pi^\prime}}$ and ${\nu_{\pi^\prime}}$ of $\zeta$ and $\nu$
afford the same one dimensional representation of $\mathbb Z/m_{\pi^\prime}\mathbb Z \cong 
\mathbb Z/n_{\pi^\prime}\mathbb Z$.

Clearly, this is an equivalence relation on the set $\Upsilon$. We denote the equivalence class containing the pair
$(m,\zeta)$ by
$[m, \zeta]$ and write $\Upsilon_{\pi^\infty}$ for the set of equivalence classes. It is also clear that each equivalence class
contains a unique pair $(n,\nu)$ where $n$ is a $\pi^\prime$-number. Hence the set $[\Upsilon_{\pi^\infty}]$ of all pairs $(m, \zeta)$ with $m_\pi = 1$ is a full set of representatives of equivalence classes on $\Upsilon_{\pi^\infty}$.

With this notation, we first determine the $\pi^\infty$-fibered biset subfunctor of $\crc$ generated by a simple biset subfunctor of it.
\begin{theorem}\label{thm:main1}
Let $(m, \zeta)\in\Upsilon$ and let $T_{\mathbb Z/m\mathbb Z, \C_\zeta}$ be the $\pi^\infty$-fibered biset subfunctor of $\crc$
generated by the image of $S_{\mathbb Z/m\mathbb Z, \C_\zeta}$ in $\crc$. Then there is an isomorphism of biset functors
\[
T_{\mathbb Z/m\mathbb Z, \C_\zeta} \cong \bigoplus_{(n,\nu)\in [m,\zeta]}S_{\mathbb Z/n\mathbb Z, \mathbb C_\nu}.
\]
\end{theorem}
\begin{proof}
To simplify the notation, we put $\zn = \mathbb Z_n$ for any $n\in\mathbb N$ throughout the proof.
Let 
\[
S_{[m,\zeta]} = \bigoplus_{(n,\nu)\in [m,\zeta]} S_{\znn, \mathbb C_\nu}
\]
be the right hand side of the above isomorphism. By its definition, $S_{[m,\zeta]}$ is a biset functor.
We have to prove that $S_{[m,\zeta]} \cong T_{\mathbb Z_m, \C_\zeta}$ as biset functors.

We first prove that $S_{[m, \zeta]}\subseteq T_{\mathbb Z_m,\mathbb C_\zeta}$. Let $(n,\nu)\in[m,\zeta]$. We shall show that the simple biset subfunctor $S_{\ZZ_n, \C_\nu}$ 
of $\mathbb C R_\mathbb C$ is contained in $T_{\mathbb Z_m,\mathbb C_\zeta}$. 
Since $m$ is a $\pi^\prime$-number, there is a $\pi$-number $k$ such that $n = mk$. We write
\[
\tilde\nu = \tilde\nu_\pi\times \tilde\nu_{\pi'}
\] 
where $\nu_\pi\in \crc(\mathbb Z_k)$ and $\nu_{\pi'}\in \crc(\mathbb Z_m)$. By the definition of the equivalence relation
on $\Upsilon$, 
we may assume that the $\pi^\prime$-part $\nu_{\pi^\prime}$ of $\nu$ coincides with $\zeta$. On the other hand, the $\pi$-part $\tilde\nu_\pi$
of $\tilde\nu$ is a virtual character of $\mathbb Z_k$,  and hence it is a complex linear combination of the 
irreducible characters of $\mathbb Z_k$, say
\[
\tilde{\nu}_{\pi}=\sum_{\chi\in\mbox{\rm\tiny{Irr}}(\mathbb Z_k)}c_{\chi}{\chi}
\]
for some complex numbers $c_\chi$. Moreover, since $k$ is a $\pi$-number, the group $\mathbb Z_k$ embeds in $A$ and 
hence each irreducible character $\chi$ of 
$\mathbb Z_k$ induces a twist biset $\Tw_{\mathbb Z_k}^\chi$. Thus, by Corollary \ref{cor:actionTw}, putting
\[
\Tw_{\nu_\pi} = \sum_{\chi}c_\chi \Tw_{\mathbb Z_k}^\chi,
\]
we obtain 
\[
\tilde{\nu}_{\pi}= \Tw_{\nu_\pi}\cdot 1
\]
where $1$ denotes the trivial character of the group $\mathbb Z_k$. Now by Lemma \ref{lem:actionCoprime}, we have
\[
\tilde\nu = \tilde\nu_\pi \times \tilde\nu_{\pi'} = (\tw_{\nu_\pi}\cdot 1) \times \tilde\zeta = (\tw_{\nu_\pi} \times \id)\cdot (1\times \tilde\zeta) =
\big((\tw_{\nu_\pi} \times \id) \infl_{\mathbb Z_m}^{\mathbb Z_n}\big) \tilde\zeta.
\]

In particular, $\tilde\nu$ is contained in the $A$-fibered biset subfunctor generated by $\tilde\zeta$. We already know that 
the simple biset functor $S_{\ZZ_n, \C_\nu}$ is generated by $\tilde\nu$. Thus 
\[
S_{\ZZ_n, \C_\nu} \subseteq T_{\mathbb Z_m,\mathbb C_\zeta}
\]
for all $(n,\nu)\in[m,\zeta]$ and hence
\[
S_{[m,\zeta]}\subseteq T_{\mathbb Z_m,\mathbb C_\zeta}
\]
as required.

To prove the reverse inclusion, it is sufficient to show that any simple biset subfunctor of $T_{\mathbb Z_m,\mathbb C_\zeta}$ is 
parameterized by a pair equivalent to $(m,\zeta)$. Indeed, since $\mathbb C R_\mathbb C$ is semisimple as a biset functor, the 
subfunctor $T_{\mathbb Z_m,\mathbb C_\zeta}$ is also semisimple as a biset functor and hence it is the 
sum of its simple subfunctors. In particular, any simple biset subfunctor of $T_{\mathbb Z_m,\mathbb C_\zeta}$ is of the form $S_{\ZZ_n, \C_\nu}$ for some $(n, \nu)\in \Upsilon$. 

Let $S_{\ZZ_n, \C_\nu} \subseteq T_{\mathbb Z_m,\mathbb C_\zeta}$. We have to show that $(n,\nu)$ is equivalent to $(m,\zeta)$.
Since $S_{\ZZ_n, \C_\nu}(\znn)\subseteq T_{\mathbb Z_m,\mathbb C_\zeta}(\mathbb Z_n)$ and 
the functor $T_{\mathbb Z_m,\mathbb C_\zeta}$ is generated by $\tilde\zeta$, we must have
\[
\tilde\nu = \gamma\cdot\tilde\zeta
\]
for some $\gamma \in B^A(\znn, \zmm)$. Then by Equation \ref{thm1:decomposition}, we deduce that $\tilde{\nu}$ is a $\mathbb{C}$-linear combination of elements of 
the form
\begin{displaymath}
\textrm{Ind}_{P}^{\znn} \textrm{Tw}_{P}^{{\phi}_{1}} \textrm{Inf}_{P/K}^{P} \Isom_{P/K,Q/L}^{\eta} \textrm{Def}_{Q/L}^{Q} \textrm{Tw}_{Q}^{{\phi}_{2}} \textrm{Res}_{Q}^{\zmm} \ \tilde{\zeta}
\end{displaymath}
for some $L\le Q\le \zmm$ and $K\le P\le \znn$ such that $\eta: Q/L\rightarrow P/K$ is an isomorphism and $\phi_1
\in P^A$ and $\phi_2\in Q^A$. But $\mathbb Z_m$ is a minimal group for the functor $T_{\mathbb Z_m,\mathbb C_\zeta}$. 
Thus the maps factoring through a group of smaller order annihilates $\tilde\zeta$ and hence any transitive summand of 
$\gamma$ must be of the form
\begin{displaymath}
\textrm{Ind}_{\mathbb Z_s}^{\znn} \textrm{Tw}_{\mathbb Z_s}^{\phi} \textrm{Inf}_{\zmm}^{\mathbb Z_s}\Isom_{\zmm, \zmm}^{\eta} \ \tilde{\zeta}
\end{displaymath}
where $\eta\in \Aut(\zm)$ and $m | s, s | n$ and $\phi : \mathbb Z_s \rightarrow A$ is a group homomorphism. Also, since $\znn$ is a minimal group for the biset functor 
$S_{\ZZ_n, \C_\nu}$, the element $\tilde\nu$ is not in the image of induction maps and hence there must exist terms where 
$s = n$, that is, terms of the form 
\begin{equation}\label{loc1}
\textrm{Tw}_{\znn}^{\phi} \textrm{Inf}_{\zmm}^{\znn}\Isom_{\zmm, \zmm}^{\eta} \ \tilde{\zeta}.
\end{equation}
Let $n = n_{\pi}n'm'$ where $n_\pi$ is the $\pi$-part of $n$, and the $\pi'$-part of $n$ satisfies $n_{\pi'} 
= n'm'$ such that $(m, n') = 1$. With this notation, we write $\znn = \mathbb Z_{n_\pi}\times 
\mathbb Z_{n'}\times \mathbb Z_{m'}$
and hence by Lemma \ref{lem:direct-product}, we get
\[
\infl_{\mathbb Z_m}^{\mathbb Z_n}\Isom_{\zmm, \zmm}^{\eta} = \infl_{1\times 1\times\mathbb Z_m}^{\mathbb Z_{n_\pi}\times
\mathbb Z_{n'}\times \mathbb Z_{m'}} \Isom_{\zmm, \zmm}^{\eta}= \infl_{1}^{\mathbb Z_{n_\pi}} \times \infl_{1}^{\mathbb Z_{n'}}\times (\infl_{\mathbb Z_m}^{\mathbb Z_{m'}}\Isom_{\zmm, \zmm}^{\eta})
\]
and 
\[
\tw_{\mathbb Z_n}^\phi = \tw_{\mathbb Z_{n_\pi}}^{\phi_\pi}\times \id \times \id
\]
where $\phi_\pi$ is the restriction of $\phi$ to the $\pi$-part of the group $\mathbb Z_n$. Also we can regard $\tilde\zeta$ as
an element of $\crc(1\times 1\times \mathbb Z_m)$ in the obvious way.
Then (\ref{loc1}) becomes
\begin{eqnarray*}
\textrm{Tw}_{\znn}^{\phi} \textrm{Inf}_{\zmm}^{\znn}\Isom_{\zmm, \zmm}^{\eta} \ \tilde{\zeta} &=&\big( (\tw_{\mathbb Z_{n_\pi}}^{\phi_\pi}\times \id \times \id) 
(\infl_{1}^{\mathbb Z_{n_\pi}} \times \infl_{1}^{\mathbb Z_{n'}}\times (\infl_{\mathbb Z_m}^{\mathbb Z_{m'}}\Isom_{\zmm, \zmm}^{\eta}))\big) (\tilde\zeta) \\ &=& 
\big( (\tw_{\mathbb Z_{n_\pi}}^{\phi_\pi}\infl_{1}^{\mathbb Z_{n_\pi}})\times
(\infl_{1}^{\mathbb Z_{n'}}) \times (\infl_{\mathbb Z_m}^{\mathbb Z_{m'}}\Isom_{\zmm, \zmm}^{\eta})\big) (\tilde\zeta) \\
&=& \phi_\pi \times 1 \times (\infl_{\mathbb Z_m}^{\mathbb Z_{m'}}\Isom_{\zmm, \zmm}^{\eta}\tilde\zeta) 
= \infl_{\mathbb Z_{n_\pi}\times 1\times \mathbb Z_m}^{\mathbb Z_n}(\phi_\pi \times 1 \times \Isom_{\zmm, \zmm}^{\eta}\tilde\zeta)\\
&=&\zeta(\eta) \infl_{\mathbb Z_{n_\pi}\times 1\times \mathbb Z_m}^{\mathbb Z_n}(\phi_\pi \times 1 \times \tilde\zeta).
\end{eqnarray*}
Here the last equality holds since the action of $\zmm^\times$ on $\tilde\zeta$ is by $\zeta$. 
In particular the element given in (\ref{loc1}) is in the image of inflation maps unless $n'm' = m$. Hence $\gamma$
must have a summand with $m = n'm'$. But this is possible only if $m$ is the $\pi'$-part of $n$ and hence all the 
summands of $\gamma$ must be of the form 
\begin{equation}\label{loc2}
\textrm{Tw}_{\znn}^{\phi} \textrm{Inf}_{\zmm}^{\znn}\Isom_{\zmm, \zmm}^{\eta} \ \tilde{\zeta} 
= (\textrm{Tw}_{\mathbb Z_{n_\pi}}^{\phi}\times \Isom_{\zmm, \zmm}^{\eta}) \ (1\times \tilde{\zeta}).
\end{equation}
In particular, we obtain that the $\pi'$-parts of $n$ and $m$ are equal. Moreover, we obtain
\[
\tilde\nu = \gamma\cdot \tilde\zeta = \sum (\textrm{Tw}_{\mathbb Z_{n_\pi}}^{\phi}\times \Isom_{\zmm, \zmm}^{\eta}) \ (1\times \tilde{\zeta})
\]
which clearly shows that the $\pi'$-parts of $\nu$ and $\zeta$ induces the same representation of $\zmm$. 
Hence we have proved that
\[
T_{\mathbb Z_m,\mathbb C_\zeta} \subseteq S_{[m,\zeta]}.
\]
This completes the proof. 
\end{proof}

As remarked above, each equivalence class $[m, \zeta]$ is represented by its unique member for which $m$ is a $\pi'$-number. In particular, 
any other member is of the form $(n, \nu)$ with $n = m\cdot k$ where $k$ is a $\pi$-number and hence the group $\mathbb Z/m\mathbb Z$
is minimal for the functor $T_{\mathbb Z/m\mathbb Z, \C_\zeta}$ and we have
\[
T_{\mathbb Z/m\mathbb Z, \C_\zeta} (\mathbb Z/m\mathbb Z) \cong \C_\zeta
\]  
where $\C_\zeta$ is the $\mathbb C\Aut(\mathbb Z/m\mathbb Z)$-module with character $\zeta$. We further have the following identification.
\begin{theorem}
The $\pi^\infty$-fibered biset functor $T_{\mathbb Z/m\mathbb Z, \mathbb C_\zeta}$ is simple and isomorphic to the functor $S_{\mathbb Z/m\mathbb Z, 1, 1, \mathbb C_\zeta}^{\pi^\infty}.$
\end{theorem}
\begin{proof}
Let $0\subset T\subseteq  T_{\mathbb Z/m\mathbb Z, \mathbb C_\zeta}$ be a $\pi^\infty$-fibered biset subfunctor of $
T_{\mathbb Z/m\mathbb Z, \mathbb C_\zeta}$. Then since $T_{\mathbb Z/m\mathbb Z, \mathbb C_\zeta}$ is
semisimple as a biset functor, there is a pair $(n, \nu)$ equivalent to $(m, \zeta)$ such that
$S_{\mathbb Z/n\mathbb Z, \mathbb C_\nu}\subseteq T$. Then by the proof of the previous theorem, we have
\[
\tilde\nu = {\tilde\nu_\pi} \times\ \tilde{\zeta}.
\]
Now since $\tw_{\phi}\tw_{\phi^{-1}} = \id$ for any $\phi\in G^{\pi^\infty}$ and $\Def^H_{H/N}\Inf_{H/N}^H = \id$ for any 
groups $N\unlhd H$, we get
\[
\tilde \zeta = \Def^{\zn}_{\zm}(\tw_{\nu_\pi^{-1}} \times\id)\tilde\nu.
\]
In particular, $\tilde \zeta$ is in $T(\zm)$ and since $T_{\mathbb Z/m\mathbb Z, \mathbb C_\zeta}$ is generated
by $\tilde\zeta$, we get $T=T_{\mathbb Z/m\mathbb Z, \mathbb C_\zeta}$. Hence $T_{\mathbb Z/m\mathbb Z, \mathbb C_\zeta}$ is simple. The second part follows from the remark above and the parametrization of simple $A$-fibered biset functors.
\end{proof}

With this proof we have also completed the proof of Theorem \ref{thm:intromain1} as the $\pi^\infty$-fibered biset 
functor $\crc$ is a sum of its simple subfunctors. 

As an immediate corollary we obtain the following result describing the restriction of the simple $\pi^\infty$-fibered biset functors that appear above to the category of biset functors.

\begin{coro}\label{cor:restSimp}
Let $A$ be as above and $(m, \zeta)\in [\Upsilon_{\pi^\infty}]$. There is an isomorphism of biset functors
\[
S_{\zm, 1, 1, \zeta}^{\pi^\infty} \cong  \bigoplus_{(n,\nu)\in [m,\zeta]} S_{\zn, \mathbb C_\nu}.
\] 
\end{coro}

Another immediate corollary is the case where $\pi$ contains all prime numbers. Then we can replace $\pi^\infty$ by the unit group $\mathbb C^\times$ of complex numbers. By the definition of $\pi$-equivalence it is clear that there is only one equivalence class which can be represented by $(1, 1)$. Hence $\crc$ is isomorphic to the simple $\mathbb C^\times$-fibered biset functor $S_{1, 1, 1, 1}$, recovering the first part of \cite[Theorem 11.3]{bolcay}.

As another special case let $p$ be a prime number and $p'$ be the set of all primes $q\neq p$. Also let $\mathbb F$ be an algebraically closed field of characteristic $p$. Then the group $A$  containing all $q^n$-th roots of unity for all 
$n\in\mathbb N$ and $q\in p'$ can be identified with the torsion subgroup of $\mathbb F^\times$. Thus $(p')^\infty$-fibered biset functors are equivalent to $\mathbb F^\times$-fibered biset functors. In particular, the 
above theorem gives a decomposition of $\crc$ as an $\mathbb F^\times$-fibered biset functor. 

We denote by $T$ the biset functor of $p$-permutation modules, see \cite{Bau} for details. We also write 
$\mathbb C T=\mathbb C\otimes T$. By \cite[Section 11D]{bolcay}, $T$ is is an $\mathbb F^\times$-fibered
biset functor, and it has a unique simple quotient isomorphic to $S_{1, 1, 1, 1}^{\mathbb F^\times}$. Together 
with Corollary \ref{cor:restSimp}, we obtain the following result.
Note that a more general result appears in \cite[Corollary 44]{Bau}.
\begin{coro}{\rm(Baumann)}\label{cor:p-perm}
 For any $m\in \mathbb N$ with $(m, p) = 1$ and any primitive character $\zeta$ modulo $m$, the simple biset 
functor $S_{\zm, \C_\zeta}$ is a composition functor of the biset functor $\mathbb C T$ of $p$-permutation modules.
\end{coro}
\subsection{Small fiber group}\label{sec:mainSmall}

Next we consider the case where $A$ is a 
finite cyclic $p$-group for a fixed prime number $p$. General case can be treated similarly. We follow the
same steps as in the previous section and determine an equivalence relation on the set $\Upsilon$ so that
the equivalence classes corresponds to $A$-fibered biset subfunctors of $\crc$. This case turns out to be more involved since we have a new distinguished fibered biset, namely $f_{(K, \kappa)}$ for $(K, \kappa)\in
\mathcal R_G$.

Throughout this section, fix a prime number $p$, a natural number $l$ and let $A$ be the group of $p^l$-th
roots of unity in $\mathbb C$. We still write $\Upsilon$ for the set of all pairs $(m, \zeta)$ defined previously.
We define a relation $\cong_A$ on $\Upsilon$ as follows. Let $(m, \zeta), (n, \nu)\in \Upsilon$. Then
$(m, \zeta)\cong_A (n, \nu)$ if 
\begin{enumerate}
\item $n_{p'} = m_{p'}$ and $\mathbb C_{\nu_{p'}}\cong \mathbb C_{\zeta_{p'}}$ and
\item either $n_p, m_p \le |A|$ or $n_p = m_p$ and for each faithful $\alpha\in A^*$, there is a complex number $c_\alpha$ such that
\begin{equation}\label{eqn:equiv2}
f_{(A, \alpha)}\cdot\tilde\nu_p = c_\alpha f_{(A,\alpha)}\cdot\tilde\zeta_p.
\end{equation}
\end{enumerate}
It is straightforward to check that this relation is an equivalence relation on $\Upsilon$. Also we shall see later that
if Equation \ref{eqn:equiv2} holds for some $\alpha$, then actually it holds for all faithful characters in $A^*$.
We again write 
$[m, \zeta]$ for the equivalence class containing the pair $(m, \zeta)$. In contrast with the previous case,
there are two types of equivalence classes in $\Upsilon/\cong_A$. Given $m\in \mathbb N$, if $m_p \le |A|$,
then 
\[
[m,\zeta] = \{(p^km_{p'} , \nu): k\le l, \mathbb C_{\nu_{p'}} \cong \mathbb C_{\zeta_{p'}}\}
\] 
and otherwise if $m_p > |A|$, then
\[
[m,\zeta] = \{(m, \nu): \mathbb C_{\nu_{p'}} \cong \mathbb C_{\zeta_{p'}} \,\mathrm{and} \, 
f_{(A, \alpha)}\cdot\tilde\nu_p =c_\alpha f_{(A,\alpha)}\cdot\tilde\zeta_p\,\mathrm{for\, some\, faithful}\, \alpha\in A^*\}.
\] 

We write $[\Upsilon_p]$ for a set of representatives of equivalence classes of the first type. As in the previous case, it can be identified with the set of all pairs $(m, \zeta)$ where $m$ is a $p'$-number. We also write $[\Upsilon_{>|A|}]$ for a set of representatives of equivalence classes of the second type. 

As in the previous section, for a given pair $(m,\zeta)\in \Upsilon$, we write $T_{\mathbb Z/m\mathbb Z, \mathbb C_\zeta}^A$ for the $A$-fibered biset subfunctor of $\crc$ generated by its biset subfunctor
$S_{\mathbb Z/m\mathbb Z, \mathbb C_\zeta}$. With this notation, we have
\begin{theorem}\label{thm:main2}
For any $(m, \zeta)\in\Upsilon$, there is an isomorphism of biset functors
\[
T_{\mathbb Z/m\mathbb Z, \mathbb C_\zeta}^A \cong \bigoplus_{(n, \nu)\in [m,\zeta]} S_{\mathbb Z/n\mathbb Z, \mathbb C_\nu}.
\]
\end{theorem}
\begin{proof}
Given $(m, \zeta)\in\Upsilon$. We let $S_{[m,\zeta]}$ be 
the right hand side of the above isomorphism. As in the previous case, we first show that $S_{[m,\zeta]}
\subseteq T_{\mathbb Z/m\mathbb Z, \mathbb C_\zeta}^A$. 

If $m_p\le |A|$, then the proof is almost identical to the proof of the corresponding part of Theorem 
\ref{thm:main1}. We skip the details. 

For the second case, suppose $m_p > |A|$ and let $(m, \nu)\cong_A (m, \zeta)$ so that 
\[
\mathbb C_{\nu_p'} \cong \mathbb C_{\zeta_p'}\, \mathrm{and}\, f_{(A,\alpha)}\cdot\tilde\nu_p = c_\alpha 
f_{(A,\alpha)}\cdot \tilde\zeta_p
\]
for all faithful $\alpha\in A^*$ where $c_\alpha\in\mathbb C^\times$. We have to prove that $\tilde\nu$ can
be written as a linear combination of elements obtained from $\tilde\zeta$ by actions of fibered bisets. 

Fix a faithful character $\alpha\in A^*$. Since the linkage class of $(A, \alpha)$ consists of all pairs 
$(A, \beta)$ with $\beta\in A^*$ faithful, the second equality above implies that 
\[
f_{\{A, \alpha\}}\cdot\tilde\nu_p = \sum_{\beta\in A^*: \scriptsize{\mathrm{faithful}}} c_\beta f_{(A,\beta)}
\cdot\tilde\zeta_p. 
\]
Also since sum of the idempotents $f_{(K, \kappa)}$ is $1$ and by Corollary \ref{lem:actionIpots},
$f_{(K, \kappa)}$ annihilates all primitive characters if $(K, \kappa)$ is not reduced or if $K< A$, we obtain
\[
\tilde\nu_p =f_{\{A, \alpha\}}\cdot\tilde\nu_p = \sum_{\beta\in A^*: \scriptsize{\mathrm{faithful}}} c_\beta f_{(A,\beta)}
\cdot\tilde\zeta_p. 
\]
In particular, $\tilde\nu_p$ is a linear combination of elements obtained from $\tilde\zeta$ by actions of fibered
bisets, as required. Thus we have proved that $S_{[m, \zeta]}\subseteq T_{\mathbb Z/m\mathbb Z,\mathbb C_\zeta}^A$.

Conversely, let $S_{\mathbb Z/n\mathbb Z, \mathbb C_\nu}$ be a biset subfunctor of 
$T_{\mathbb Z/m\mathbb Z, \mathbb C_\zeta}^A$. We have to prove that $(n, \nu)\cong_A (m, \zeta)$.
By the choice of $(n, \nu)$, there is an element $X\in B^A(\mathbb Z/n\mathbb Z, \mathbb Z/m\mathbb Z)$
such that 
\[
\tilde\nu = X\cdot \tilde\zeta.
\]
We have two cases. First suppose that $m_p \le |A|$. Then we can represent the equivalence class
$[m, \zeta]$ by the pair $(m, \zeta)$ where $(m, p) = 1$. 

By Equation \ref{eqn:partial}, any transitive summand of $X$ can be written as
\begin{displaymath}
\textrm{Ind}_{P}^{\zn} \textrm{Inf}_{P/K}^{P} Y \,\textrm{Def}_{Q/L}^{Q} \textrm{Res}_{Q}^{\zm}
\end{displaymath}
for some $L\le Q\le \zm$ and $K\le P\le \zn$. Here $Y$ is a transitive $A$-fibered $(Q/L, P/K)$-
biset with stabilizing pair $(U, \phi)$ such that $p_1(U) = P/K, p_2(U) = Q/L$ and restrictions of $\phi$ to
$k_1(U)$ and $k_2(U)$ are both faithful. As in the previous case of Theorem \ref{thm:main1}, since both 
$\zeta$ and $\nu$ are primitive, any summand of $X$ where $|P/K| < n$ annihilates 
$\tilde\zeta$ and, also $\tilde\nu$ is not in the image of any transitive biset of the above form if $|Q/L| < m$. Hence $Y$ can be chosen as a linear combination of transitive $A$-fibered $(\zn, \zm)$-bisets with 
stabilizing pairs $(U, \phi)$ such that  $p_1(U) = \zn, p_2(U) = \zm$ and restrictions of $\phi$ to
$k_1(U)$ and $k_2(U)$ are both faithful. 

Given such a summand $Y$, by Lemma \ref{lem:direct-product}, we write $Y = Y_{p'}\times Y_p$ where $Y_p\in 
B^A(\mathbb Z/n_p\mathbb Z, 1)$ and $Y_{p'}\in B^A(\mathbb Z/n_{p'}\mathbb Z, \mathbb Z/m\mathbb Z)$.
Furthermore, since $A$ is a $p$-group, the group $B^A(\mathbb Z/n_{p'}\mathbb Z, \mathbb Z/m_{p'}\mathbb Z)$ can be identified by the Burnside group of $(\mathbb Z/n_{p'}\mathbb Z,\zm)$-bisets. In particular,
$Y_{p'}$ is a $(\mathbb Z/n_{p'}\mathbb Z, \mathbb Z/m\mathbb Z)$-biset which must be an isogation since
restrictions of $\phi$ to the corresponding subgroups must be faithful. Hence we get $n_{p'} = m_{p'}$ and 
$\mathbb C_{\nu_{p'}}\cong \mathbb C_{\zeta_{p'}}$ after identifying $\zm = \mathbb Z/n_{p'}\mathbb Z$.  

On the other hand, the above conditions imply that the stabilizing pair for $Y_p$ is of the form $(U_p, 
\phi_p)$ such that $p_1(U_p) = \mathbb Z/n_{p}\mathbb Z, p_2(U_p) = \mathbb Z/m_{p}\mathbb Z$ and 
restrictions of $\phi$ to $k_1(U_p)$ and $k_2(U_p)$ are both faithful.
Since $(m, p) =1$, then we necessarily have $(U_p, \phi_p) =((\mathbb Z/n_{p}\mathbb Z)\times 1, 
\phi\times 1)$ for some faithful $\varphi: \mathbb Z/n_{p}\mathbb Z \to A$. Hence, in this case, we must have 
$n_p\le |A|$ and $Y_p = \tw_{\mathbb Z/n_{p}\mathbb Z}^\varphi\infl_{1}^{\mathbb Z/n_{p}\mathbb Z}$.
Thence we have proved that the $p'$-parts of $n$ and $m$ are equal and $n_p, m_p\le |A|$, as required.

For the second case, if $m_p > |A|$, then we also have $n_p > |A|$. The condition on $p'$-parts can be 
obtained in the same way as the previous case. Hence, by Lemma \ref{lem:actionCoprime}, we can drop the indexes and write $m= m_p$ and $n =n_p$ and $Y = Y_p$. Also let $(U, \phi)$ be a stabilizing pair for $Y$. Then by 
Theorem \ref{thm:factorization}, we can factor $Y$ as
\[
[Y] = \tw_{\zn}^{\tilde\phi} \infl_{(\zn)/K}^{\zn}\Isom_{(\zn)/K, \zm}^{\tilde\eta}e_{(L,\phi_L)}
\]
where the notation is chosen according to Theorem \ref{thm:factorization}. Now suppose $K\not = 1$. Then
since $n>|A|$, the homomorphism $\tilde\phi: \zn\to A$ has a non-trivial kernel, say $K'$. If $K' > K$, then
$\tilde\phi$ can be regarded as a homomorphism $(\zn)/K\to A$ and hence we have
\[
[Y] = \infl_{(\zn)/K}^{\zn}\tw_{(\zn)/K}^{\tilde\phi}\Isom_{(\zn)/K, \zm}^{\tilde\eta}e_{(L,\phi_L)}.
\]
On the other hand, if $K'\le K$, then $(\zn)/K \cong \frac{(\zn)/K'}{K/K'}$, hence we have
\[
[Y] = \infl_{(\zn)/K'}^{\zn}\tw_{(\zn)/K'}^{\tilde\phi}\infl^{(\zn)/K'}_{(\zn)/K}\Isom_{(\zn)/K, \zm}^{\tilde\eta}e_{(L,\phi_L)}.
\]
In both cases, $Y\cdot \tilde\zeta$ is in the image of inflation maps. But by its choice $\tilde\nu$ is not in this image. 
Hence there is at least one transitive term $Y$ such that $K = 1$. But then since $\zm\cong (\zn)/K$, we must have $n = m$ and hence
any summand $Y$ must be of the form 
\[
[Y] = \tw_{\zm}^{\psi}\Isom_{\zm, \zm}^{\sigma}e_{(L,\lambda)}
\]
for some $\psi\in G^A, \sigma\in\Aut(\zm)$ and $\lambda\in L^A$. Therefore, we get
\[
\tilde\nu = \sum_{\psi, \sigma, \lambda} c_{\psi, \sigma, \lambda} \tw_{\zm}^{\psi}\Isom_{\zm, \zm}^{\sigma}e_{(L,\lambda)}\tilde\zeta
\]
where $c_{\psi, \sigma, \lambda}\in \mathbb C$ and the sum is over all $\psi\in G^A, \sigma\in\Aut(\zm), 
\lambda\in L^A$. 

Furthermore, by Corollary \ref{lem:actionIpots}, $e_{(L,\lambda)}\cdot\tilde\zeta = 0$ unless $L = A$. When $L = A$, 
then $e_{(A, \lambda)} = f_{(A, \lambda)}$ and we put $\tilde\zeta_\alpha = f_{A, \lambda}\cdot\tilde\zeta$. Also by 
Corollary \ref{cor:actionIso} and by Lemma \ref{lem:actionTw}, we obtain
\[
\tilde\nu = \sum_{\psi, \sigma, \lambda} c_{\psi, \sigma, \lambda} \tilde\zeta_{\psi\cdot{}^\sigma\alpha}
\]
Finally since $f_{(A,\alpha)}\cdot\tilde\zeta_\beta =(f_{(A,\alpha)}\cdot f_{(A,\beta)})\cdot\tilde\zeta_\beta = 0$ unless $\alpha = \beta$ and $f_{(A,\alpha)}\cdot\tilde\zeta_\alpha = \tilde\zeta_\alpha$, multiplying the above equality by $f_{(A, \alpha)}$ we get
\[
f_{(A,\alpha)}\cdot\tilde\nu = \sum_{\psi, \sigma, \lambda: \alpha = \phi\cdot{}^\sigma\alpha} c_{\psi, \sigma, \lambda} f_{(A,\alpha)}\cdot\tilde\zeta = c_\alpha f_{(A,\alpha)}\cdot\tilde\zeta
\]  
for some non-zero complex number $c_\alpha$. Note that the same argument shows that if $\beta\in A^*$
is another faithful character, then $f_{A,\beta}\cdot\tilde\nu  = c_\beta f_{A,\beta}\cdot\tilde\zeta$ for some
non-zero complex number $c_\beta$. 

This completes the proof of the last claim. Hence we have shown that $(n, \nu)\cong_A (m, \zeta)$, and
hence $T_{\zm, \mathbb C_\zeta}^A = S_{[m,\zeta]}$.
\end{proof}

As in the previous case, the functor $T_{\zm,\mathbb C_\zeta}^A$ is simple. Proof is almost identical to the 
previous case, we skip details.
To identify it, first assume $(m, \zeta)\in [\Upsilon_p]$. Then as in the case of large enough fiber group, we have an isomorphism
\[
T_{\zm, \mathbb C_\zeta}^A \cong S_{\zm, 1,1, \mathbb C_\zeta}^A.
\]
Note that although the quadruple parametrizing this simple functor is the same as the one from the previous case, evaluations of the corresponding simple functors are not the same for large groups.

On the other hand, if $(m, \zeta)\in [\Upsilon_{>|A|}]$, then, by its construction, we have 
\[
T_{\zm, \mathbb C_\zeta}^A (\zm) \cong \bigoplus_{\nu: (m, \nu) \cong_A (m, \zeta)} \mathbb C_\nu 
\]
as $\mathbb C\Aut(\zm)$-modules. Fix a faithful character $\alpha$ of $A$. 
In $T_{\zm, \mathbb C_\zeta}^A (\zm)$, the one-dimensional subspace generated by $\tilde\zeta_\alpha\times \zeta_{p'} = (f_{(A, \alpha)}\times 1) \cdot \tilde\zeta$ is
$\Gamma_{(\zm, A, \alpha)}$-invariant. Indeed, any element of this group is of the form 
 $\Tw^\phi_{\zm}\Isom^\sigma_{\zm,\zm}e_{(A, \alpha)}\times 1$ where we write $1$ for the identity biset for
the $p'$-part of $\zm$, for some $\phi\in \zm^A$ and $\sigma\in \Aut(\zm)$ such that 
 $\phi\cdot {}^\sigma\alpha = \alpha$. Thus given such an element 
$\Tw^\phi_{\zm}\Isom^\sigma_{\zm,\zm}e_{(A, \alpha)}$
in $\Gamma_{(\zm, A, \alpha)}$, we get, by Corollary \ref{cor:actionIso} and Lemma \ref{lem:actionTw}, that
\[
\Tw^\phi_G\Isom^\sigma_{G,G}e_{(A, \alpha)} \cdot \tilde\zeta_\alpha = \rho_{\zeta, \alpha}(\phi, \sigma)\tilde\zeta_\alpha.
\]
for some complex number $\rho_{\zeta, \alpha}(\phi, \sigma)$. Let $\mathbb C_\zeta^\alpha$ be this one-dimensional $\mathbb C\Gamma_{(\zm, A, \alpha)}$-module. Then
clearly, 
\[
\mathbb C_\zeta^\alpha = (f_{(A, \alpha)}\times 1)\cdot T_{\zm, \mathbb C_\zeta}^A (\zm) \quad \mathrm{and}\quad 
 T_{\zm, \mathbb C_\zeta}^A (\zm) =\bar E_{\zm}\otimes_{\mathbb C\Gamma_{(\zm, A, \alpha)}} \mathbb C_\zeta^\alpha.
\]
Therefore by \cite[Theorem 9.2]{bolcay}, the simple functor $T_{\zm, \mathbb C_\zeta}^A$ is parameterized by the quadruple $(\zm, A, \alpha, \mathbb C_\zeta^\alpha)$. This completes the proof of the following theorem.
\begin{theorem} Let $p$ be a prime number and $A$ be a finite cyclic $p$-group. Then 
\begin{enumerate}
\item[(i)] The $A$-fibered biset subfunctor $T_{\zm, \mathbb C_\zeta}^A$ of $\mathbb C R_\mathbb C$ is simple. 
\item[(ii)] If $(m, \zeta)\in[\Upsilon_p]$, then $T_{\zm, \mathbb C_\zeta}^A$ is isomorphic to the simple functor $S_{\zm, 1, 1, \mathbb C_\zeta}^A$.
\item[(iii)]  If $(m, \zeta)\in[\Upsilon_{>|A|}]$, then $T_{\zm, \mathbb C_\zeta}^A$ is isomoriphic to the simple functor $S_{\zm, A, \alpha, \mathbb C_\zeta^\alpha}^A$.
\end{enumerate}
\end{theorem}
With this theorem, we have completed the proof of Theorem \ref{thm:intromain2}. The general case where the fibre group $A$ is of composite order, or where is arbitrary subgroup of $\CC^\times$ can be obtained easily by combining results from Theorem \ref{thm:intromain1} and Theorem \ref{thm:intromain2}.

Again as a corollary, we obtain the following result on restriction of simple fibered biset functors that appear above.
\begin{coro}
Let $A$ be a finite cyclic $p$-group. Then
\begin{enumerate}
\item For any $(m, \zeta)\in [\Upsilon_p]$, there is an isomorphism of biset functors
\[
S_{\zm, 1, 1, \mathbb C_\zeta}^A \cong \bigoplus_{(n,\nu)\in [m,\zeta]} S_{\zn, \mathbb C_\nu}.
\]
\item For any $(m, \zeta)\in [\Upsilon_{> |A|}]$, there is an isomorphism of biset functors
\[
S_{\zm, A, \alpha, \mathbb C_\zeta^\alpha}^A \cong \bigoplus_{(n,\nu)\in [m,\zeta]} S_{\znn, \mathbb C_\nu}.
\]
\end{enumerate}
\end{coro} 

\end{document}